\documentclass[11pt,oneside,reqno]{amsart}
\usepackage{amssymb}
\usepackage{color,xcolor}
\usepackage{}
\usepackage{amsmath}
\usepackage{graphicx}
\usepackage{mathrsfs}
\usepackage{amsfonts}
\usepackage{enumerate,amsmath,amssymb,amsthm}
\usepackage{hyperref}
\usepackage{accents}

\usepackage{comment}

\numberwithin{equation}{section}

\newcommand{\be}{\begin{eqnarray}}
\newcommand{\ee}{\end{eqnarray}}
\newcommand{\ce}{\begin{eqnarray*}}
\newcommand{\de}{\end{eqnarray*}}
\newtheorem{theorem}{Theorem}[section]
\newtheorem{lemma}[theorem]{Lemma}
\newtheorem{remark}[theorem]{Remark}
\newtheorem{definition}[theorem]{Definition}
\newtheorem{proposition}[theorem]{Proposition}
\newtheorem{Examples}[theorem]{Example}
\newtheorem{corollary}[theorem]{Corollary}
\newtheorem{assumption}{Assumption}[section]

\def\d{\delta}

\def\[{{\Big[}}
\def\]{{\Big]}}
\def\<{{\langle}}
\def\>{{\rangle}}
\def\({{\Big(}}
\def\){{\Big)}}

\def\bx{{\mathbf{x}}}

\def\min{{\mathord{{\rm min}}}}

\def\={&\!\!=\!\!&}
\def\bt{\begin{theorem}}
\def\et{\end{theorem}}
\def\bl{\begin{lemma}}
\def\el{\end{lemma}}
\def\br{\begin{remark}}
\def\er{\end{remark}}
\def\bas{\begin{assumption}}
\def\eas{\end{assumption}}
\def\bd{\begin{definition}}
\def\ed{\end{definition}}
\def\bp{\begin{proposition}}
\def\ep{\end{proposition}}
\def\bc{\begin{corollary}}
\def\ec{\end{corollary}}
\def\bx{\begin{Examples}}
\def\ex{\end{Examples}}

\def\EE{{\mathbb E}}

\def\R{{\mathbb R}}

\def\d{\text{\rm{d}}}

\def\geq{\geqslant}
\def\leq{\leqslant}

\allowdisplaybreaks
\usepackage{color}

\allowdisplaybreaks

\title[Stochastic logarithmic Schr\"odinger equations driven by L\'evy noise]{\bf{Stochastic logarithmic Schr\"odinger equations driven by L\'evy noise}$^\dagger$}

  \thanks{$\dagger$  This work is partially supported by National Key R\&D Program of China(Nos.  2024YFA1012301, 2022YFA1006001), National Natural Science Foundation of China (Nos.12071433, 12131019, 11971456, 11721101), and the Fundamental Research Funds for the Central Universities (No. WK3470000031).}

\date{\today}
\begin{document}

\maketitle

\centerline{\scshape    Jiahui Zhu$^a$ Jianliang Zhai$^{b}$ }
\medskip
 {\footnotesize
 \centerline{ $a.$    School of Mathematical Sciences, Zhejiang University of Technology, Hangzhou 310019, China }
  \centerline{$b.$ School of Mathematical Sciences, University of Science and Technology of China, Hefei 230026, China}

 }

\begin{abstract}
In this paper, we study the stochastic logarithmic Schr\"odinger equation with saturated nonlinear multiplicative L\'evy noise. The global well-posedness is established for the stochastic logarithmic Schr\"odinger equation in an appropriate Orlicz space by construct  solutions of a regularized equation converging strongly to a solution to the original equation.


\end{abstract}

\section{Introduction}

The deterministic logarithmic Schr\"odinger equation
\begin{align}\label{det-LSE}
i \partial_t u(t,x)+\frac{1}{2} \Delta u(t,x)+\lambda \log (|u(t,x)|^2) u(t,x)=0, \quad u(0,x)=u_0,\quad (t,x)\in\R\times \R^d,
\end{align}
 was initially proposed in  \cite{BM-76} as a model of quantum mechanics, notable for being the only type to satisfy the tensorization property. The logarithmic nonlinearity was subsequently acknowledged as suitable  for describing various physical phenomena, including nuclear physics \cite{Hefter-85}, quantum optics \cite{KEB-2000},  transport and diffusion phenomena \cite{MFCL-03}, open quantum systems \cite{Yasue}, effective quantum gravity \cite{Zlo-10}, and Bose–Einstein condensation \cite{AZ-11}.

The logarithmic Schr\"odinger equation \eqref{det-LSE} demonstrates rather unique dynamical properties. These distinct behaviors are primarily driven by the singularity of the logarithm at the origin. Observe that regardless of the sign of $\lambda$, the energy 
\begin{align}\label{Intro-energy-eq1}
E(u(t)):=\frac{1}{2}\|\nabla u(t)\|_{L^2\left(\mathbb{R}^d\right)}^2-\frac{\lambda }{2}\int_{\mathbb{R}^d}|u(t, x)|^2 \log |u(t, x)|^2 d x
\end{align}
has no definite sign. When $\lambda>0$, no solutions exhibit dispersive structure. Whereas for  $\lambda<0$, solutions do disperse and  the standard dispersion rate in $t^{-d / 2}$, is enhanced by a logarithmic factor in $(t \sqrt{\ln t})^{-d / 2}$, resulting in an accelerated dispersion (see \cite{CG-Duke-18}). Therefore, we classify the case $\lambda<0$ as defocusing.  Another feather of \eqref{det-LSE} is that  the modulus of the solution converges for large time to a universal Gaussian profile.

 The logarithmic nonlinearity, rather than being a weak nonlinearity, dominates the dynamics even in the presence of an energy-subcritical defocusing power  nonlinearity.   In other words, the effects of this  logarithmic nonlinear term are significantly stronger than those of a defocusing power-like nonlinearity.  All the positive Sobolev norms of solution of the logarithmic equation grow logarithmically in time, even when an  energy-subcritical defocusing power-like nonlinearity is added, as shown in \cite{CG-Duke-18}.
 
 Driven by the aforementioned physical significance, our primary focus here is  the well-posedness of the following stochastic nonlinear logarithmic Schr\"odinger equation on $\mathbb{R}^d$
\begin{align}\label{SLogSE}
&  \d u(t)=\mathrm{i}[\Delta u(t)+\lambda u(t) \log |u(t)|^2]\d t-\mathrm{i}   \sum_{j=1}^m g_j(u(t-)) \diamond \mathrm{d} L_j(t)=0,\quad t>0,\\
&u(0)=u_0,\nonumber
\end{align}
where $\lambda \in \mathbb{R} \backslash\{0\}$ and $L(t)=\left(L_1(t), \cdots, L_m(t)\right), t \geq 0$ is an $\mathbb{R}^m$-valued pure jump L\'evy process with L\'evy measure $\nu$, i.e. $L(t)=\int_0^t \int_B z \tilde{N}(\mathrm{d} s, \mathrm{d} z)$, where $B=\left\{z \in \mathbb{R}^m: 0<\right.$ $|z|_{\mathbb{R}^m} \leq 1\}$ and $N$ is a time homogeneous Poisson random measure on $\mathbb{R}^{+} \times(\mathbb{R}^m\backslash\{0\})$  with $\sigma$-finite intensity measure $\nu$ satisfying $\int_B|z|_{\mathbb{R}^m}^2 \nu(\mathrm{d} z)<\infty$  on a filtered probability space $(\Omega,\mathcal{F},\mathbb{F}:=(\mathcal{F}
_t)_{t\geq 0},\mathbb{P})$ with the usual conditions. 

The nonlinearity $g_j$ is of saturated type $g_j(y)=\tilde{g}_j(|y|^2) y$ with $\tilde{g}_j \in C_b^2([0, \infty) ; \mathbb{R})$. 
 In optics, saturated nonlinearities are often used to describe propagation of intense beams through saturable media. Under Assumption \ref{assum-main}, $\tilde{g}$ could be  the following typical functions:  saturation of the intensity nonlinearity $\tilde{g}_j(\theta)=\frac{\theta}{1+\rho \theta}$, with $\rho>0$, which characterizes photorefractive media \cite{Kelley} and square-root nonlinearity $\tilde{g}_j(\theta)=1-\frac1{\sqrt{1+\theta}}$, which describe narrow-gap semiconductors \cite{SBB}.  We can also choose $ \tilde{g}_j(\theta)=\frac{\theta(2+\rho \theta)}{(1+\rho \theta)^2}$ or  $\tilde{g}_j(\theta)=\frac{\log (1+\rho \theta)}{1+\log (1+\rho \theta)}$,  $\theta\geq 0$, with $\rho>0$.

In view of the expression \eqref{Intro-energy-eq1}, the natural energy space is defined by
\begin{align*}
W:=\big\{u \in H^1(\mathbb{R}^d):|u|^2 \log  |u|^2 \in L^1(\mathbb{R}^d)\big\},
\end{align*}
which can be represented as a Banach space using Orlicz spaces. 

In the deterministic case, the Cauchy problem for equation \eqref{det-LSE} in the energy space $W$ was initially studied in \cite{Caz-80}, where it was demonstrated that if $\lambda>0$ and $u_0 \in W$, then there exists a unique solution $u \in C(\mathbb{R}, W)$ to \eqref{det-LSE}. The proof in \eqref{det-LSE} is based on compactness arguments and  the regularization of the nonlinearity, which seems to be restricted to $\lambda>0$. An alternative proof presented in \cite{HO-24}, adopted a more constructive approach, demonstrating the strong convergence of a sequence of approximate solutions while avoiding reliance on compactness arguments.

Recently, in \cite{CG-Duke-18}, weighted Sobolev spaces $\mathcal{F}(H^{\alpha})$ in $L^2$ defined by
$$\mathcal{F}(H^\alpha):=\{u \in L^2(\mathbb{R}^d), x \mapsto\langle x\rangle^\alpha u(x) \in L^2(\mathbb{R}^d)\},\quad \langle x\rangle:=\sqrt{1+|x|^2}$$ 
with the norm $\|u\|_{\mathcal{F}(H^{\alpha})}:=\| \langle x\rangle^\alpha u(x)\|_{L^2}$ for some $\alpha \in(0,1]$, are considered in order to control the nonlinearity in the region $\{|u|<1\}$. The existence result in weighted Sobolev spaces was proved in \cite{CG-Duke-18} that if $u_0 \in H^m(\mathbb{R}^d) \cap \mathcal{F}(H^\alpha)$ with $m=1,2$ and $\alpha \in(0,1]$, then there exists a unique solution $u \in L_{\text {loc }}^{\infty}(\mathbb{R}, H^m(\mathbb{R}^d) \cap\mathcal{F}(H^\alpha))$. 
It is important to note that weighted Sobolev spaces $\mathcal{F}(H^\alpha)$ is strictly narrower than the energy space $W$:
\begin{align}\label{inc-rel-H1-W}
H^1(\mathbb{R}^d) \cap \mathcal{F}(H^\alpha) \subsetneq W \subsetneq H^1(\mathbb{R}^d),
\end{align}
for all $\alpha\in(0,1]$, see \cite{HO-24} for more discussions.

In the stochastic case, by using the rescaling  transformation and maximal monotone operators approach, Barbu, R\"ockner and Zhang \cite{BRZ17} proved global existence and uniqueness of solutions of the stochastic logarithmic Schr\"odinger equation with linear multiplicative Gaussian noise in the energy space $W$. However, the approach appears to be effective only for the case of linear multiplicative noise. Recently, Cui and Sun in \cite{CS-23} studied the stochastic logarithmic Schr\"odinger equation with either additive noise or nonlinear multiplicative Gaussian noise and established the global well-posedness of the solution in the weighted Sobolev space $ H^1(\mathbb{R}^d) \cap \mathcal{F}(H^\alpha)$ by using a regularization of the nonlinearity. 
In particular, the space $ H^1(\mathbb{R}^d) \cap \mathcal{F}(H^\alpha)$ in \cite{CS-23} allows the authors to effectively control the nonlinear terms as
\begin{align*}
  \| u\log|u|^2 \|_{L^2(\R^d)}\lesssim_{\delta} \|u\|_{L^{2-2\delta}(\R^d)}^{1-\delta}+\|u\|_{L^{2+2\delta}(\R^d)}^{1+\delta}\lesssim \| u\|^{1-\delta-\frac{d\delta}{2\alpha}}_{L^2(\R^d)}\| u\|_{\mathcal{F}(H^{\alpha})}^{\frac{d\delta}{2\alpha}}+\|u\|_{L^2(\R^d)}^{1+\delta-\frac{d\delta }{2}}\|\nabla u\|_{L^2(\R^d)}^{\frac{d\delta }{2}},
\end{align*}
for $\alpha>\frac{d\delta}{2-2\delta}$, $u\in H^1(\R^d)\cap \mathcal{F}(H^{\alpha})$.

Compared to Gaussian cases \cite{BRZ17,CS-23},  stochastic logarithmic Schr\"odinger equations driven by jump noise remain unexplored. The unique challenges posed by L\'evy noise arise from its discontinuous nature, which distinguishes the analysis from that of  stochastic nonlinear Schr\"odinger equations driven by Gaussian noise. In  Gaussian cases, Stratonovich integrals are typically employed, but they prove inadequate for handling jump noise. To address this, the Marcus canonical integral is crucial, as it is specifically designed to accommodate discontinuities while maintaining invariance under changes of coordinates, a property necessary for conserving norms in stochastic nonlinear Schr\"odinger equations. However, the presence of an additional discontinuous component, associated with Marcus-type noise, introduces a novel layer of complexity, making the analysis of our problem \eqref{SLogSE} particularly challenging.

Recognizing the inclusion relation \eqref{inc-rel-H1-W} in functions spaces, in this paper, our aim is to establish the global existence and uniqueness of solutions to equation \eqref{SLogSE} in the general energy space $W$,  comparable to the results found in \cite{CS-23}. 
The well-posedness of the Cauchy problem for \eqref{SLogSE} is more intricate than it may seem at first glance, with the primary difficulty stemming from the lack of locally Lipschitz continuity of the nonlinear term due to the singularity of the logarithm at the origin. The standard approach, which relies on deterministic and stochastic Strichartz estimates, and a fixed point argument on Duhamel’s formula, as used for the power nonlinearities employed in \cite{BLZ}, is impractical.  But luckily, the logarithmic nonlinearity satisfies the following inequality:
\begin{align}\label{quasi-mon-intro}
|\operatorname{Im}[(u \log |u|-v \log |v|)(\bar{u}-\bar{v})]| \leq|u-v|^2 \quad \text { for all } u, v \in \mathbb{C}.
\end{align}
To avoid the blowup of the nonlinear term at the origin, we will adopt a regularization approach for the logarithmic nonlinearity and introduce the regularized logarithmic Schr\"odinger equation 
\begin{align}\label{intro-app-SlogSE}
&  \d u(t)=\mathrm{i}\Big[ \Delta u(t)+2 \lambda u(t) \log \Big(\frac{|u(t)|+\varepsilon}{1+\varepsilon|u(t)|}\Big)\Big]\d t-\mathrm{i}   \sum_{j=1}^n g_j(u(t-)) \diamond \mathrm{d} L_j(t), \;t>0, \\
&u(0)=u_0,\nonumber
\end{align}
where $0<\varepsilon< 1$. One advantage of this approximation is that it works whichever the sign of $\lambda$ is, while the approach introduced initially in \cite{Caz-80} seems to be bound to the focusing case $\lambda>0$.
Another advantage of our choice of regulation function $\log  \Big(\frac{|u|+\epsilon}{1+\epsilon|u|}\Big)$, different from \cite{CG-Duke-18} and \cite{HO-24}, is that it is bounded and enables us to obtain similar inequality for the regularized logarithmic nonlinearity. 

Unlike \cite{BRZ17}, which relies on compactness arguments, we will construct solutions by demonstrating that the approximate solutions of the regularized equation \eqref{intro-app-SlogSE} form a Cauchy sequence in the Banach space $L^2(\Omega ; L^{\infty}([0, T] ; L^2_{\text{loc}}(\R^d)))$. This approach provides strong convergence directly, without resorting to weak convergence or taking subsequences. Notably, the Marcus form of the noise does not disrupt the cancellation effects, which leads to mass conservation similar to the deterministic case. We establish a uniform $H^1$-estimate (independence of $\varepsilon$) of the approximate equations resulting to the global well-posedness in $H^1(\R^d)$ of \eqref{SLogSE}. To control the singularity of the logarithmic nonlinearity at the origin, we analyze the entropy function and establish the uniform estimates in the Orlicz space which subsequently yield uniform estimates in the energy space $W$.

We now state our main result.

\begin{theorem}\label{Them-main}Suppose that Assumption \ref{assum-main}  holds.  Let $\lambda \in \mathbb{R} \backslash\{0\}$ and $p\geq 2$. Let $0<T<\infty$. 
If $u_0\in L^p(\Omega,\mathcal{F}_0;W)$, then there exists a unique global solution $u\in L^p(\Omega;L^{\infty}(0,T;W))$ to equation \eqref{SLogSE}  in the sense of Definition \ref{defi-solution-W}. 
Moreover, there exists a constant $C(u_0,\lambda,m,T,p,K_{\tilde{g}})>0$ such that
\begin{align*}
\mathbb{E} \sup_{t\in[0,T]}\|u(t)\|_{W}^p\leq C(u_0,\lambda,m,T,p,K_{\tilde{g}}),
\end{align*}
and the mass conservation holds $\mathbb{P}$-a.s. for all $t\in[0,T]$,
\begin{align}\label{Th-mass-conv-1}
\|u(t)\|_{L^2(\mathbb{R}^d)}=\|u_0\|_{L^2(\mathbb{R}^d)}.
\end{align}
\end{theorem}

The rest of this paper is structured as follows. In section \ref{sec-pre} we introduce some appropriate spaces and assumptions. Section \ref{sec-approx-eq} is mainly concerned with the regularized approximating equation. We establish the global well-posedness of the approximating equation in $L^2(\R^d)$. In section \ref{sec-uni-est},  obtain the uniform estimate in $H^1(\R^d)$ as well as in the Orlicz space for the solutions of the approximating equations. The proof of the main result is given in Section \ref{sec-proof-mian}. Some technical details are postponed to the Appendix. 

Throughout this article, $C$ denotes some nonnegative constant that may change from line to line. To indicate the dependence of the constant on the data, we write $C = C(a, b,\cdot,\cdot)$ with the understanding that this dependence increases in its variables. 
For two non-negative numbers $A$ and $B$, we write $A \lesssim B$ to indicate that $A \leq C B$ for some $C>0$ that may change from line to line, and we may write $\lesssim_\delta$ if the implicit constant depends on $\delta$.

\section{Preliminaries}\label{sec-pre}

 Let $L^p(\mathbb{R}^d;\mathbb{C})$ (or $L^p(\mathbb{R}^d)$, when there is no risk of confusion) denote the space of all $p$-integrable complex functions $u: \mathbb{R}^d \rightarrow \mathbb{C}$ with the norm $\|u\|_{L^p}$.  
We  equip $H:=L^2(\mathbb{R}^d;\mathbb{C})$  with the scalar product
\begin{align*}
\langle u, v\rangle_{L^2}=\operatorname{Re} \int_{\mathbb{R}^d} u(x) \overline{v(x)}\d  x.
\end{align*}
Note that $H$ is a real Hilbert space with this product and this product is topologically equivalent to the standard complex inner product.

Let $H^m(\mathbb{R}^d, \mathbb{C})$ (or simply $H^m(\mathbb{R}^d)$) be the classical Sobolev spaces equipped with their usual norms $\|\cdot\|_{H^{m}}$. Especially, 
$H^1(\mathbb{R}^d)=\big\{u \in L^2(\mathbb{R}): \nabla u \in L^2(\mathbb{R})\big\}$ with norm $\|u\|_{H^1}^2=\|u\|_{L^2}^2+\|\nabla u\|_{L^2}^2$. 

The functional of energy in \eqref{Intro-energy-eq1} is generally neither finite nor of class $C^1$ on $H^1(\R^d)$. Due to this lack of smoothness,  it is convenient to work in a suitable Banach space when studying the existence of solutions to \eqref{SLogSE}, as this ensures that the functional of energy is well defined and $C^1$ smooth.

We define the functions $F$ and $N$ on $\mathbb{R}^{+}$ by
\begin{align}
& F(s)=-s^2 \log (s)^2, \label{defini-F-functiom} \\
& N(s)= \begin{cases}-s^2 \log \left(s^2\right) & \text { if } 0 \leq s \leq e^{-3}, \\
3 s^2+4 e^{-3} s-e^{-6} & \text { if } s \geq e^{-3}.\end{cases} \label{defini-A-functio}
\end{align}

We note that $N$ is a convex and increasing $C^1$-function, which is $C^2$ and positive except at origin:
$N \in C^1([0,+\infty)) \cap C^2((0,+\infty)).$
 We define the Orlicz space $V$ corresponding to $N$ by
\begin{align*}
V=\big\{\varphi \in L_{\text {loc }}^1(\mathbb{R}^d): N(|\varphi|) \in L^1(\mathbb{R}^d)\big\},
\end{align*}
and the Luxembourg norm by
\begin{align*}
\|\varphi\|_{V}=\inf \Big\{k>0: \int_{\mathbb{R}^d} N\Big(\frac{|\varphi(x)|}{k}\Big)\d x \leq 1\Big\},
\end{align*}
where $L^1_{\text{loc}}$ is the space of all locally Lebesgue integrable functions. According to \cite{Caz-83}, 
 $(V,\|\cdot\|_{V})$ is a separable reflexive Banach space. For any $\varphi\in V$, we have
 \begin{align}\label{norm-in-Orlicz-space}
      \min\{\|\varphi\|_V,\|\varphi\|_V^2\}\leq \int_{\mathbb{R}^d} N(|\varphi(x)|)\d x\leq \max \{\|\varphi\|_V,\|\varphi\|_V^2\}.
 \end{align}

Define $$W:=\{\varphi \in H^1(\mathbb{R}^d):F(|\varphi|) \in L^1(\mathbb{R}^d)\}.$$
Then  $W=H^1(\mathbb{R}^d) \cap V$ and  $W$ is a reflexive Banach space equipped with the norm $\|\varphi\|_{W}:=\|\varphi\|_{H^1}+\|\varphi\|_{V}$, see \cite[Proposition 2.2]{Caz-83}. We have the following continuous embedding:
\begin{align*}
      W \hookrightarrow L^2(\R^d) \hookrightarrow W',
\end{align*}
where $W'=H^{-1}(\mathbb{R}^d)+V'$ is the dual space of $W$ equipped with the norm $$\|u\|_{W'}=\inf\{\|u_1\|_{H^{-1}}+\|u_2\|_{V'}:\,  u=u_1+u_2,\; u_1\in H^{-1}(\mathbb{R}^d),\, u_2\in V'    \},\quad u\in W'.$$

\begin{lemma}\cite[Lemma 2.6]{Caz-83}\label{lem-Caz-bounded-L}
The operator $L: u \mapsto \Delta u+u \log |u|^2$ maps continuously from $W $ to $W'$. The image under $L$ of a bounded subset of $W$ is a bounded subset of $W'$. The operator $E$ defined by \eqref{Intro-energy-eq1} belongs to $C^1(\mathbb{R}, W)$ and
\begin{align*}
E^{\prime}(u)=L u-u
\end{align*}
for all $u \in W$.
\end{lemma}

For $0<T<\infty$, let us denote $\mathbb{D}(0,T;L^2(\mathbb{R}^d))$ the space of all right continuous functions with left-hand limits from $[0,T]$ to $L^2(\mathbb{R}^d)$ and $Y_{T}:=L^{\infty}(0,T;L^2(\mathbb{R}^d))$.  
Note that $\mathbb{D}(0,T;L^2(\mathbb{R}^d))$ is a Banach space with the uniform norm, but not separable.  Let $M^p_{\mathbb{F}}(Y_{T}):=M^p_{\mathbb{F}}(\Omega;L^{\infty}(0,T;L^2(\mathbb{R}^d)) )$ denote the space of all $L^2(\mathbb{R}^{d})$-valued $\mathbb{F}$-progressively measurable processes $u:[0,T]\times \Omega \rightarrow L^2(\mathbb{R}^d)$  satisfying
  \begin{align*}
  \|u\|^p_{M^p_{\mathbb{F}}(Y_{T})}:=\mathbb{E}\|u\|^p_{ L^{\infty}(0,T;L^2(\mathbb{R}^d)) }=\mathbb{E}\Big(\sup_{s\in[0,T]}\|u(s)\|^{p}_{L^2}\Big)<\infty.
  \end{align*}

\begin{assumption}\label{assum-main}
For each $1 \leq j \leq m$, there exists a function $\tilde{g}_j \in C_b^2([0, \infty) ; \mathbb{R})$  such that $g_j$ is given by $$g_j(y)=\tilde{g}_j\left(|y|^2\right) y,\quad y \in \mathbb{C}.$$ Assume that $\tilde{g}_j:[0, \infty) \rightarrow \mathbb{R}$ is continuously differentiable and satisfies
\begin{align}\label{Ass-g-boundedness-1}
\max_{1\leq j\leq m}\sup _{\theta>0} \big[\tilde{g}_j(\theta)+(1+\theta)\tilde{g}_j^{\prime}(\theta)+(1+\theta^{\frac32})\tilde{g}_j^{\prime\prime}(\theta)\big]\leq K_{\tilde{g}}<\infty.
\end{align} 
\end{assumption}

Now let us provide a precise description of the Marcus canonical integral $\diamond$ as it appears in equation \eqref{SLogSE}. 
Define a generalized Marcus mapping 
$$
 \Phi: \mathbb{R}_{+} \times \mathbb{R}^m \times \mathbb{C} \rightarrow \mathbb{C}
 $$
such that for each fixed $z\in\mathbb{R}^m$, $x\in \mathbb{C}$, the function
$$t\mapsto\Phi(t,z, x)$$
is the unique solution of the following equation
$$
\frac{\partial \Phi}{\partial t}(t, z, x)=-\mathrm{i} \sum_{j=1}^m z_j g_j(\Phi(t, z, x)), \quad \Phi(0, z, x)=x .
$$
Observe that since functions $g_j$ are continuously differentiable and $g_j'$ is uniformly bounded,   similar arguments as given in \cite[Lemma 2]{Marcus-81} shows that the map $\Phi$ is well defined.
Let us clarify that $\Phi$ is considered as a function of the last variable for fixed $t$ and $z$. Denote $$\Phi(z, \cdot) := \Phi(1, z, \cdot).$$
Then equation \eqref{SLogSE} with notation $\diamond$ is defined by
\begin{align}\label{SLSE-1}
\mathrm{d} u(t)=&\mathrm{i} \left[\Delta u(t)+\lambda u(t) \log |u(t)|^2\right] \mathrm{d}t +\int_B[\Phi(z, u(t-))-u(t-)] \tilde{N}(\mathrm{d} t, \mathrm{d} z)\nonumber \\
& +\int_B\Big[\Phi(z, u(t))-u(t)+\mathrm{i} \sum_{j=1}^m z_j g_j(u(t))\Big] \nu(\mathrm{d} z) \mathrm{d} t, \quad t> 0,
\end{align}
with the same initial condition $u(0)=u_0$ as in \eqref{SLogSE}, see \cite{Marcus-78, Marcus-81} for more details. 

In the following remark, we provide insight into the advantages of employing  Marcus canonical integrals.
\begin{remark}\label{remark-marcus-int}
The Marcus canonical integral $\diamond$ has several beneficial consequences and can be viewed as a generalization Stratonovich
integrals for continuous case. One notable feature is its consistency with a Wong-Zakai-type approximation.  More precisely, the jump in the differential can be approximated sensibly by piecewise linear continuous functions (see \cite{Marcus-78} and \cite{KPP-95} for details). Additionally, the Marcus integral ensures the validity of the change of variables formula. 
Another significant advantage  is its invariance under  coordinate transformations. This invariance result was first proven in \cite{Marcus-78} and later extended  to infinite dimensional settings in \cite{Brz+Man-19}. This property is crucial for ensuring the mass conservation of solutions for stochastic nonlinear Schr\"odinger equations, see e.g. \cite[proof of Proposition 4.1]{BLZ}. 
\end{remark}

The precise definition of solutions to \eqref{SLSE-1} is given below.

\begin{definition}\label{defi-solution-W} Let $0<T<\infty$. 
An $L^2(\R^d)$-valued $\mathbb{F}$-adapted process $u$ is said to be a global solution to \eqref{SLogSE} if 
\begin{enumerate}
    \item[(i)] $u\in L^p(\Omega;L^{\infty}(0,T;W))$;
    \item[(ii)]$u\in \mathbb{D}(0,T;L^2(\mathbb{R}^d))$ $\mathbb{P}$-a.s.;
    \item[(iii)]  it satisfies $\mathbb{P}$-a.s. for every $t\in[0,T]$,
\begin{align}
    u(t)
          &=\int_0^t \mathrm{i}\Delta u(s) \d s + \int_0^t  \mathrm{i}\lambda u(s)\log |u(s)|^2 \d s+ \int_0^t \int_B  \Phi(z, u(s-))-u(s-)  \tilde{N}(\d s,\d z) \nonumber\\
      &\ \ \ + \int_0^t \int_B \Phi(z, u(s))-u(s)+\mathrm{i} \sum_{j=1}^m z_j g_j(u(s))  \nu(\d z)\d s,\label{eq-def-u-H-1}
\end{align}
in $W'$.
\end{enumerate}

\end{definition}

\begin{remark}As a byproduct of Theorem \ref{Them-main}, by assuming the initial 
    $u_0\in L^p(\Omega,\mathcal{F}_0;H^1(\R^d))$ instead of $u_0\in L^p(\Omega,\mathcal{F}_0;W)$, our proof also establishes the global existence and uniqueness of the $H^1$-solution to equation \eqref{SLogSE}. More precisely, if we assume that $u_0\in L^p(\Omega,\mathcal{F}_0;H^1(\R^d))$, 
    then there exists a unique global solution $u$ in $ L^p(\Omega;L^{\infty}(0,T;H^1(\R^d) ))$ to \eqref{SLogSE}  such that $u\in \mathbb{D}(0,T; L^2_{\textrm{loc}}(\mathbb{R}^d))$ $\mathbb{P}$-a.s. and it satisfies the equation \eqref{eq-def-u-H-1} in the sense of $H^{-1}(D)$, for all bounded open sets $D\subset \R^d$.
\end{remark}

\section{Approximating equation}\label{sec-approx-eq}

To neutralize the singularity of the logarithm at the origin, we consider the following approximate equation with a small regularized parameter $0< \varepsilon< 1$, 
\begin{align}\label{app-SlogSE}
&  \d u_{\varepsilon}(t)=\mathrm{i}[ \Delta u_{\varepsilon}(t)+2 \lambda u_{\varepsilon}(t) L_{\varepsilon}(u_{\varepsilon}(t))]\d t-\mathrm{i}   \sum_{j=1}^n g_j(u_{\varepsilon}(t-)) \diamond \mathrm{d} L_j(t),  \\
&u_{\varepsilon}(0)=u_0,\nonumber
\end{align}
where 
\begin{equation}
L_{\varepsilon}(u)=\log \left(\frac{|u|+\varepsilon}{1+\varepsilon|u|}\right), \quad \forall u \in \mathbb{C}.
\end{equation}
Based on the definition of the Marcus canonical integral in Section \ref{sec-pre}, the above equation with the notation $\diamond$ is defined as follows
\begin{align}
\mathrm{d} u_{\varepsilon}(t)= & \mathrm{i}\left[\Delta u_{\varepsilon}(t)+2\lambda u_{\varepsilon}(t)L_{\varepsilon}(u_{\varepsilon}(t))\right] \mathrm{d} t+\int_B[\Phi(z, u_{\varepsilon}(t-))-u_{\varepsilon}(t-)] \tilde{N}(\mathrm{d} t, \mathrm{d} z) \nonumber \\
& +\int_B\Big[\Phi(z, u_{\varepsilon}(t))-u_{\varepsilon}(t)+\mathrm{i} \sum_{j=1}^m z_j g_j(u_{\varepsilon}(t))\Big] \nu(\mathrm{d} z) \mathrm{d} t, \quad t>0.\label{app-SlogSE-marcus}
\end{align}

Let $(S_t)_{t\in\R}$ denote the group of isometries on $L^2(\R^d)$ generated by the operator $\text{i}\Delta$.  
\begin{definition}\label{defi-mild-solution-app}Let $0<T<\infty$.  
An $L^2(\R^d)$-valued $\mathbb{F}$-adapted process $u_{\varepsilon}$ is said to be a global mild solution to equation \eqref{app-SlogSE-marcus} if
\begin{enumerate}
\item[(i)]$u_{\varepsilon}\in  \mathbb{D}(0,T;L^2(\mathbb{R}^d))$ $\mathbb{P}$-a.s.;
\item[(ii)]$u_{\varepsilon}$ belongs to $M^p_{\mathbb{F}}(Y_{T})$;
\item[(iii)] it satisfies $\mathbb{P}$-a.s.
  \begin{align}
     u_{\varepsilon}(t)=&S_tu_0+\mathrm{i}\int_0^tS_{t-s}\big(2\lambda u_{\varepsilon}(s) L_{\varepsilon}(u_{\varepsilon}(s))\big)\d s\nonumber\\
     &+\int_0^t\int_BS_{t-s}\Big[\Phi(z,u_{\varepsilon}(s-))-u_{\varepsilon}(s-)\Big]\tilde{N}(\d s,\d z)\label{SEE-trucated-1}\\
     &+\int_0^t\int_BS_{t-s}\Big[\Phi(z,u_{\varepsilon}(s))-u(s)+\mathrm{i}\sum_{j=1}^mz_jg_j(u_{\varepsilon}(s))\Big]\nu(\d z)\d s,\quad \text{for }0\leq t\leq T.\nonumber
\end{align}
\end{enumerate}
\end{definition}

In the following Lemma, we highlight several properties of $L_{\varepsilon}$, with the proof provided in  the Appendix. 
\begin{lemma}\label{lem-property-L_varepsilon} Let $0<\varepsilon< 1$ and 
\begin{equation}
L_{\varepsilon}(u)=\log \left(\frac{|u|+\varepsilon}{1+\varepsilon|u|}\right), \quad \forall u \in \mathbb{C}.
\end{equation} Then the following assertions hold:
\begin{enumerate}
\item[(a)] For all $u\in\mathbb{C}$, $|L_{\varepsilon}(u)| \leq|\log \varepsilon|$, and $\big||u| L_{\varepsilon}(u)\big| \leq\big||u| \log|u| \big|$.
\item[(b)] For all $u_1, u_2 \in \mathbb{C}$,
\begin{align*}
\left|u_1 L_{\varepsilon}\left(u_1\right)-u_2 L_{\varepsilon}\left(u_2\right)\right| \leq(1+\log (1 / \varepsilon))\left|u_1-u_2\right| .
\end{align*}
\item[(c)]For $\delta\in(0,1)$, there exist $C(\delta)>0$ such that for all $u_1, u_2 \in \mathbb{C}$ and $\varepsilon,\mu\in(0,1)$, 
\begin{align}\label{L-var-est-2}
&\left|\operatorname{Im}\left(\overline{u}_1-\overline{u}_2\right)\left(u_1 L_{\varepsilon}(u_1)-u_2 L_{\mu}(u_2)\right)\right| \nonumber\\&\leq(1-\varepsilon^2)|u_1-u_2|^2+C |\varepsilon-\mu   ||u_1-u_2| +C(\delta)| \varepsilon-\mu|^{\delta}  |u_2|^{1+\delta} |u_1-u_2|.
\end{align}
\item[(d)]For $\alpha,\delta \in(0,1)$, there exist $C(\alpha),C(\delta)>0$ such that for all $u_1, u_2 \in \mathbb{C}, \varepsilon \in(0,1)$,
\begin{align}
&|u_1 L_\varepsilon(u_1)-u_2 \log | u_2| |\nonumber\\
& \leq  \varepsilon+C(\delta)\varepsilon^{\delta}|u_1|^{1+\delta} +|u_2-u_1|\nonumber\\
&\quad\quad+C(\alpha) \left(1+|u_2|^{1-\alpha} \log ^{+}|u_2|+|u_1|^{1-\alpha} \log ^{+}|u_1|\right)|u_2-u_1|^\alpha, \label{L-var-est-3}
\end{align}
where $\log ^{+} x:=\max \{\log x, 0\}$.
\end{enumerate}
\end{lemma}

\begin{remark}

As an approximation for the nonlinear term, various regulation functions, such as $ \log (|u|^2+\varepsilon)$ in \cite{CG-Duke-18}, $\log (|u|+\varepsilon)^2 $ in \cite{HO-24}, have been proposed in the literature. The advantage of our choice of regulation function $2L_{\varepsilon}(u)=\log  \left(\frac{|u|+\epsilon}{1+\epsilon|u|}\right)^2$ is that it is bounded and  enabling us to extend the inequality for the logarithmic nonlinearity
\begin{align*}
|\operatorname{Im}[(u \log |u|-v \log |v|)(\bar{u}-\bar{v})]| \leq|u-v|^2 \quad \text { for all } u, v \in \mathbb{C},
\end{align*}
to the inequality \eqref{L-var-est-2} for the regulation nonlinear terms with $\varepsilon, \mu\in(0,1)$, 
which is crucial when proving the sequence of approximate solutions to \eqref{app-SlogSE-marcus} forms a Cauchy sequence in  $L^2(\Omega ; L^{\infty}([0, T] ; L^2_{\text{loc}}(\R^d)))$. Meanwhile, esitmate  \eqref{L-var-est-3} allows us to establish the strong convergence of the approximated logarithmic nonlinear term $u_{\varepsilon}L_{\varepsilon}(u_{\varepsilon} )$. 
This approximation remians effective regardless of the sign of $\lambda$ is, whereas the approach introduce in \cite{Caz-83} appears to be limited the case $\lambda>0$. 
\end{remark}



In order to establish the global wellposedness of \eqref{app-SlogSE-marcus} and obtain uniform estimates for the solution $u_{\varepsilon}$ with respect to $0<\varepsilon < 1$, we need  to introduce some auxiliary Lemmata. We begin by examining properties of the operator $\Phi$. 
For the simplicity of notation, for each $s \in[0, 1], z \in \mathbb{R}^m, y \in \mathbb{C}$, we denote
\begin{align*}
\mathcal{G}(s, z, y) & :=\Phi(s, z, y)-y, \\
\mathcal{H}(s, z, y) & :=\Phi(s, z, y)-y+\mathrm{i} \sum_{j=1}^m z_j g_j(y).
\end{align*}
For abbreviation, we denote $G(z, y)=\mathcal{G}(1, z, y)$ and $H(z, y)=\mathcal{H}(1, z, y)$.

\begin{lemma}\label{lem-stoch-lip-1} Under Assumption \ref{assum-main}, we have for all $y, y_1,y_2 \in \mathbb{C}$, $s\in[0,1]$ and all $z \in \mathbb{R}^m$ with $0<|z|_{\mathbb{R}^m} \leq 1$, 
\begin{align}
| \Phi(s,z, y_1)-\Phi(s,z, y_2)| & \lesssim_{m,K_{\tilde{g}}} | y_1-y_2|, \label{Phi-est-Lip}\\
| \mathcal{G}(s,z, y)| & \lesssim_{m,K_{\tilde{g}}}|z|_{\mathbb{R}^m}| y|, \\
|  \mathcal{G}(s,z, y_1)-  \mathcal{G}(s,z, y_2)| & \lesssim_{m,K_{\tilde{g}}} |z|_{\mathbb{R}^m}| y_1- y_2|,\label{G-est-Lip-1}\\
| H(z, y)|& \lesssim_{m,K_{\tilde{g}}} |z|_{\mathbb{R}^m}^2| y|, \\
|  H(z, y_1)- H(z, y_2)| & \lesssim_{m,K_{\tilde{g}}} |z|_{\mathbb{R}^m}^2| y_1- y_2|.\label{H-est-Lip-1}
\end{align}
\end{lemma}
\begin{proof}
The proof of this Lemma, with only minor adjustments and changes in constants is nearly identical to the proof of Lemma 3.3 in \cite{BLZ}, so it is omitted.
\end{proof}

\begin{lemma}\label{lem-G-H1} Under Assumption \ref{assum-main}, we have for all $u\in H^1(\R^d)$ and all $z \in \mathbb{R}^m$ with $0<|z|_{\mathbb{R}^m} \leq 1$,
\begin{align}
     \|  G(z, u)\|_{H^1} &\lesssim_{m,K_{\tilde{g}}} |z|_{\R^m}\|  u\|_{H^1},\\
\| H(z, u)\|_{H^1}& \lesssim_{m,K_{\tilde{g}}} |z|_{\mathbb{R}^m}^2\| u\|_{H^1}.  
\end{align}
\end{lemma}
\begin{proof}
For each $j$,  $g_j(u)=\tilde{g}_j(|u|^2)u$ defines a Fr\'echet differentiable map $g_j: L^2(\mathbb{R}^d) \rightarrow L^2(\mathbb{R}^d)$ and  the derivative  $g_j'(u)\in\mathcal{L}(L^2(\mathbb{R}^d),L^2(\mathbb{R}^d))$ for $u\in L^2(\mathbb{R}^d)$  which is given by
\begin{align*}
    g_j'[u]h=\tilde{g}_j(|u|^2)h+\tilde{g}_j'(|u|^2)2\operatorname{Re}( u\overline{h})u, \quad h\in L^2(\mathbb{R}^d).
\end{align*}
It follows that for $v\in H^1(\R^d)$,
\begin{align}\label{proof-nabla-g-eq1}
      \nabla g_j(v)=\nabla \big(\tilde{g}_j(|v|^2)v\big)=\tilde{g}_j(|v|^2)\nabla v+2\tilde{g}'_j(|v|^2)\operatorname{Re}(\nabla v\overline{v})v.
 \end{align}
 By using the fact that $|\operatorname{Re}(v\bar{w})|\leq |v||w|$ and Assumption \ref{assum-main}, we infer 
 \begin{align*}
      \| \nabla g_j(v)\|_{L^2}
      \leq \|\tilde{g}_j(|v|^2)\nabla v\|_{L^2}+2\|\tilde{g}'_j(|v|^2)\operatorname{Re}(\nabla v\overline{v})v\|_{L^2}
            \leq 3K_{\tilde{g}}\|\nabla v\|_{L^2}.
  \end{align*}
   Using the above estimate and the Cauchy-Schwartz inequality, we deduce
 \begin{align*}
 \|\mathcal{G}(s,z,u)\|_{H^1}=\|\Phi(s,z,u)-u\|_{H^1}&=\Big\|\int_0^s \sum_{j=1}^m -\mathrm{i} z_jg_j(\Phi(a,z,u))\d a\Big\|_{H^1}\\
 &\leq  |z|_{\R^m}\int_0^s \Big(\sum_{j=1}^{m}\|g_j(\Phi(a,z,u))\|_{H^1}^2\Big)^{\frac12}\d a\\
 &\leq 3K_{\tilde{g}} m^{\frac12} |z|_{\R^m}\int_0^s\|\Phi(a,z,u))\|_{H^1}\d a\\
  &\leq 3K_{\tilde{g}} m^{\frac12} |z|_{\R^m}\|u\|_{H^1}+3K_{\tilde{g}} m^{\frac12} |z|_{\R^m}\int_0^s\|\mathcal{G}(a,z,u)\|_{H^1}\d a.
 \end{align*}

   Applying the Gronwal Lemma yields
  \begin{align}\label{lem-est-eq-10-1}
     \|\mathcal{G}(s,z,u)\|_{H^1}&\leq M_1  |z|_{\R^m}\|u\|_{H^1},
  \end{align}
  where $M^1:=3K_{\tilde{g}} m^{\frac12}  e^{3K_{\tilde{g}} m^{\frac12} } $.  
  
 \noindent  Note that  for $y\in\mathbb{C}$,
 \begin{align}\label{id-g-deri-j-k}
  (\mathrm{i} g_j)'[y](\mathrm{i}g_k(y))=-\tilde{g}_j(|y|^2)\tilde{g}_k(|y|^2)y.
 \end{align}
  From the weak chain rule, we derive that for  $v\in H^1(\R^d)$, 
\begin{align*}
\nabla\big( \tilde{g}_j(|v|^2)  \tilde{g}_k(|v|^2)v\big) 
&=\tilde{g}_j'(|v|^2)  \tilde{g}_k(|v|^2)2\operatorname{Re}(\overline{v}\nabla v)v+\tilde{g}_j(|v|^2)  \tilde{g}_k'(|v|^2)2\operatorname{Re}(\overline{v}\nabla v)v\\
&\quad+\tilde{g}_j(|v|^2)  \tilde{g}_k(|v|^2)\nabla v.
\end{align*}
Hence by using Assumption \ref{assum-main}, we infer
\begin{align}
\Big\| \nabla\big( \tilde{g}_j(|v|^2)  \tilde{g}_k(|v|^2)v\big)\Big\|_{L^2} 
&\leq 5K_{\tilde{g}}^2 \|\nabla v\|_{L^2}.\label{proof-lem-est-phi-eq-2}
\end{align}
 We now apply the above estimate, the Cauchy-Schwartz inequality, \eqref{lem-est-eq-10-1}, and  \eqref{id-g-deri-j-k} to obtain
      \begin{align*}
      &\|H(z,u)\|_{H^1}=\|\Phi(1,z,u)-u+\mathrm{i}\sum_{j=1}^mz_jg_j(u) \|_{H^1}\\
       &=\Big\| \int_0^1  -\sum_{j=1}^m z_j \Big[\mathrm{i} g_j (\Phi(a,z,u))-\mathrm{i} g_j(u)\Big]  \d a\Big\|_{H^1} \\
       &=\Big\| \int_0^1  \sum_{j=1}^m z_j \int_0^a\frac{\d (\mathrm{i} g_j)}{\d\Phi}[\Phi(b,z,u)]\Big(-\mathrm{i}\sum_{k=1}^m z_kg_k(\Phi(b,z,u))\Big)\d b\,  \d a\Big\|_{H^1}\\
             &=\Big\| \int_0^1 \int_0^a  \sum_{j=1}^m \sum_{k=1}^m z_j  z_k\Big(\tilde{g}_j(|\Phi(b,z,u)|^2) \tilde{g}_k(|\Phi(b,z,u)|^2)\Phi(b,z,u)  \Big)\d b\,  \d a\Big\|_{H^1}\\
        &\leq |z|^2_{\R^m}\int_0^1\int_0^a \Big(\sum_{j=1}^m \sum_{k=1}^m \Big\|\tilde{g}_j(|\Phi(b,z,u)|^2) \tilde{g}_k(|\Phi(b,z,u)|^2)\Phi(b,z,u)  \Big\|^2_{H^1}  \Big)^{\frac12}\d b\, \d a\\
             &\leq 5K_{\tilde{g}}^2 m |z|^2_{\R^m}\int_0^1\int_0^a \|\Phi(b,z,u)\|_{H^1}  \d b \d a\\
        &\leq  5K_{\tilde{g}}^2 m |z|^2_{\R^m}\int_0^1\int_0^a \Big(\|u\|_{H_1}+\|\mathcal{G}(b,z,u)\|_{H^1}  \Big)\d b \d a\\
        &\leq  M_2 |z|^2_{\R^m}\|u\|_{H^1},
  \end{align*}
  where $M_2:= 5K_{\tilde{g}}^2 m (1+M_1) $. 
  
\end{proof}


\begin{lemma} \label{lem-est-nabla-Phi}Under Assumption \ref{assum-main}, we have for all $u \in H^1(\R^d)$ and all $z \in \mathbb{R}^m$ with $0<|z|_{\mathbb{R}^m} \leq 1$,
\begin{align}
\Big| \|\nabla \Phi (z,u)\|_{L^2}^p -\|\nabla u\|_{L^2}^p \Big| & \lesssim_{m,K_{\tilde{g}}} |z|_{\mathbb{R}^m}\|\nabla u\|_{L^2}^p,\label{lemma-est-nabla-Phi-eq1}\\
\Big| \|\nabla \Phi (z,u)\|_{L^2}^p -\|\nabla u\|_{L^2}^p  +p\|\nabla u\|_{L^2}^{p-2}\langle \nabla u, \nabla \mathrm{i} \sum_{j=1}^m z_j g_j(u)\rangle_{L^2} \Big|  & \lesssim_{m,K_{\tilde{g}}} |z|^2_{\mathbb{R}^m}\|\nabla u\|_{L^2}^p.\label{lemma-est-nabla-Phi-eq2}
\end{align}
\end{lemma}
\begin{proof}
Define
\begin{align}\label{func-K-def}
\mathcal{K}(u):=\|\nabla u\|_{L^2}^p=\left(\operatorname{Re}\int_{\mathbb{R}^d} \nabla u(x) \overline{\nabla u(x)} \d x\right)^{\frac{p}2}, \quad  u \in H^1(\mathbb{R}^d) .
\end{align}
Note that
\begin{align}
&\mathcal{K}^{\prime}[u]h=p\|\nabla u\|_{L^2}^{p-2} \langle \nabla u, \nabla h\rangle_{L^2}=p\|\nabla u\|_{L^2}^{p-2}\operatorname{Re}\int_{\mathbb{R}^d} \nabla u(x)  \overline{\nabla h(x)} \d x .\label{func-K-first-der}
\end{align}
Let $0\leq s\leq 1$. Since $\Phi(0,z,u)=u$, applying Assumption \ref{assum-main} yields
\begin{align*}
&\Big|\mathcal{K}( \Phi (s,z,u)) -\mathcal{K}( u)\Big| \\
&=\Big|\int_0^s\frac{d}{da}\mathcal{K}(\Phi(a,z,u))\d a=\int_0^s\mathcal{K}'[\Phi(a,z,u)]\big(-\mathrm{i} \sum_{j=1}^m z_j g_j(\Phi(a, z, u))\big)\d a\Big|\\
&=\Big|\int_0^s p\|\nabla \Phi(a,z,u)\|_{L^2}^{p-2}\langle \nabla \Phi(a,z,u), \nabla\big( -\mathrm{i} \sum_{j=1}^m z_j g_j(\Phi(a, z, u))\big)\rangle_{L^2}\d a\Big|\\
&=\Big|\int_0^s p\|\nabla \Phi(a,z,u)\|_{L^2}^{p-2}\langle \nabla \Phi(a,z,u),  -\mathrm{i} \sum_{j=1}^m z_j \tilde{g}_j(|\Phi(a, z, u)|^2)\nabla \Phi(a, z, u)\rangle_{L^2}\d a\\
&\ \ \ +\int_0^s p\|\nabla \Phi(a,z,u)\|_{L^2}^{p-2}\langle \nabla \Phi(a,z,u),  -2\mathrm{i} \sum_{j=1}^m z_j \tilde{g}_j'(|\Phi(a, z, u)|^2)\operatorname{Re}\big(\nabla\Phi(a, z, u)\overline{\Phi(a, z, u)}\big) \Phi(a, z, u)\rangle_{L^2}\d a\Big|\\
&=\Big|\int_0^s p\|\nabla \Phi(a,z,u)\|_{L^2}^{p-2}\langle \nabla \Phi(a,z,u),  -2\mathrm{i} \sum_{j=1}^m z_j \tilde{g}_j'(|\Phi(a, z, u)|^2)\operatorname{Re}\big(\nabla\Phi(a, z, u)\overline{\Phi(a, z, u)}\big) \Phi(a, z, u)\rangle_{L^2}\d a\Big|\\
&\leq 2 K_{\tilde{g}} m^{\frac12}p |z|_{\mathbb{R}^m} \int_0^s\|\nabla \Phi(a,z,u)\|_{L^2}^{p}\d a\\
&=  2K_{\tilde{g}}m^{\frac12}p|z|_{\mathbb{R}^m}s\mathcal{K}(u)+   2K_{\tilde{g}} m^{\frac12}p|z|_{\mathbb{R}^m} \int_0^s \mathcal{K}( \Phi (a,z,u)) )-\mathcal{K}(u) \d a\\
&\leq
2K_{\tilde{g}}m^{\frac12}p|z|_{\mathbb{R}^m}s\mathcal{K}(u)+   2K_{\tilde{g}} m^{\frac12}p|z|_{\mathbb{R}^m} \int_0^s\Big| \mathcal{K}( \Phi (a,z,u)) )-\mathcal{K}(u)\Big| \d a,
\end{align*}
where we used $\langle x,\mathrm{i}x\rangle_{L^2} =0$, for all $x\in L^2(\R^d)$ in the fourth identity and \eqref{Ass-g-boundedness-1} in the first inequality. 
Applying the Gronwall Lemma gives
\begin{align}
\Big|\mathcal{K}( \Phi (s,z,u)) -\mathcal{K}( u)\Big|&\leq    2K_{\tilde{g}} m^{\frac12}p|z|_{\mathbb{R}^m}s\mathcal{K}(u) \exp\Big( 2K_{\tilde{g}} m^{\frac12}p|z|_{\mathbb{R}^m} \Big)\nonumber\\
&\leq M_{3} |z|_{\mathbb{R}^m}\mathcal{K}(u),\label{lem-proof-eq-101}
\end{align}
where $M_{3}=2K_{\tilde{g}} m^{\frac12}p \exp\Big( 2K_{\tilde{g}} m^{\frac12}p\Big)$.
The proof of \eqref{lemma-est-nabla-Phi-eq1} is complete.

Similarly
\begin{align*}
&  \|\nabla \Phi (z,u)\|_{L^2}^p -\|\nabla u\|_{L^2}^p  +p\|\nabla u\|_{L^2}^{p-2}\langle \nabla u, \nabla\mathrm{i}\sum_{j=1}^m z_j g_j(u)\rangle_{L^2}   \\
&=  \mathcal{K}( \Phi (z,u)) -\mathcal{K}( u)  -\mathcal{K}'[u]\Big(- \mathrm{i} \sum_{j=1}^m z_j g_j(u)\Big)   \\
&=   \int_0^1 \frac{\partial}{\partial a}\mathcal{K}(\Phi(a,z,u))-\mathcal{K}'[u]\Big(- \mathrm{i} \sum_{j=1}^m z_j g_j(u)\Big)  \d a   \\
&=  \int_0^1 \frac{\partial}{\partial a}\mathcal{K}(\Phi(a,z,u))-\frac{\partial}{\partial a}\mathcal{K}(\Phi(a,z,u)) _{a=0} \d a  \\
&=  \int_0^1 \int_0^a \frac{\partial^2}{\partial\theta^2}\mathcal{K}(\Phi(\theta,z,u))\d \theta \d a \\
&= \int_0^1 \int_0^a \frac{\partial}{\partial\theta}\Big[\mathcal{K}'[\Phi(\theta,z,u)](-\mathrm{i} \sum_{j=1}^m z_j g_j(\Phi(\theta,z,u)))\Big]\d \theta \d a  \\
&=  \int_0^1 \int_0^a \mathcal{K}'[\Phi(\theta,z,u)]\big(-\mathrm{i} \sum_{j=1}^m z_j g'_j[\Phi(\theta,z,u)]\big(-\mathrm{i} \sum_{j=1}^m z_j g_j(\Phi(\theta,z,u))\big)\big)\d \theta \d a\\&\ \ \ + \int_0^s \int_0^a \mathcal{K}''[\Phi(\theta,z,u)]\big(-\mathrm{i} \sum_{j=1}^m z_j g_j(\Phi(\theta,z,u)),-\mathrm{i} \sum_{j=1}^m z_j g_j(\Phi(\theta,z,u))\big)\d \theta \d a \\
&:=I_1+I_2.
\end{align*}
Recall that (see \eqref{id-g-deri-j-k})
\begin{align}
g_j^{\prime}[x_1] (\mathrm{i}g_k(x_1))
&=\mathrm{i}\tilde{g}_j(|x_1|^2) \tilde{g}_k(|x_1|^2) x_1.\label{proof-lem-est-phi-eq-1}
\end{align}
For the first term, by using \eqref{proof-lem-est-phi-eq-1}, \eqref{proof-lem-est-phi-eq-2}, \eqref{func-K-first-der}, and \eqref{lem-proof-eq-101}, 
we have
\begin{align}
|I_1|&= \Big|\int_0^s \int_0^a \mathcal{K}'[\Phi(\theta,z,u)]\big(\sum_{j=1}^m(\sum_{k=1}^m z_jz_k \tilde{g}_j(|\Phi(\theta,z,u)|^2)  \tilde{g}_k(|\Phi(\theta,z,u)|^2)\Phi(\theta,z,u)\d \theta \d a\Big|\nonumber \\
&\leq |z|_{\mathbb{R}^m}^2 \int_0^s \int_0^a p\|\Phi(\theta,z,u)\|^{p-2}_{L^2}\Big|\Big\langle \nabla \Phi(\theta,z,u), \sum_{j,k=1}^m\nabla\Big( \tilde{g}_j(|\Phi(\theta,z,u)|^2)  \tilde{g}_k(|\Phi(\theta,z,u)|^2)\Phi(\theta,z,u)\Big) \Big\rangle_{L^2}\Big| \d \theta \d a\nonumber\\
&\leq 5pK^2_{\tilde{g}} |z|_{\mathbb{R}^m}^2 \int_0^s \int_0^a \|\nabla\Phi(\theta,z,u)\|^{p}_{L^2}\d \theta \d a\nonumber\\
&= 5pK^2_{\tilde{g}} |z|_{\mathbb{R}^m}^2 \int_0^s \Big[a\mathcal{K}(u)+    \int_0^a \mathcal{K}( \Phi (\theta,z,u)) )-\mathcal{K}(u) \d \theta\Big]\d a\nonumber\\
&\leq 5pK^2_{\tilde{g}} |z|_{\mathbb{R}^m}^2 (1+M_3)  \|\nabla u\|^p_{L^2} ,\label{lemma-proof-nabla-eq201}
\end{align}
where we also used the fact that $|z|_{\mathbb{R}^m}\leq1$. In view of \eqref{proof-nabla-g-eq1}, we derive
 \begin{align*}
 \Big\|  -\mathrm{i} \sum_{j=1}^m z_j \nabla g_j(v)\Big\|_{L^2}^2
 &\leq \sum_{j=1}^m|z|^2_{\R^m}\Big\|\tilde{g}_j(|v|^2)\nabla v+2\tilde{g}'_j(|v|^2)\operatorname{Re}(\nabla v\overline{v})v\Big\|_{L^2}^2\\
 &\leq 8\sum_{j=1}^m|z|^2_{\R^m}\Big(\|\tilde{g}_j(|v|^2)\nabla v\|_{L^2}^2+\|\tilde{g}'_j(|v|^2)\operatorname{Re}(\nabla v\overline{v})v\|_{L^2}^2\Big)\\
  &\leq  8mK_{\tilde{g}}^2 |z|^2_{\R^m}  \|\nabla v\|_{L^2}^2.
  \end{align*}
For the second term, it follows from the above estimate and \eqref{lem-proof-eq-101} that
\begin{align}
|I_2|
&= \Big|\int_0^s \int_0^a p\|\nabla\Phi(\theta,z,u)\|^{p-2}_{L^2}\Big\langle  -\mathrm{i} \sum_{j=1}^m z_j \nabla g_j(\Phi(\theta,z,u)), -\mathrm{i} \sum_{j=1}^m z_j \nabla g_j(\Phi(\theta,z,u))\Big\rangle_{L^2} \d \theta \d a\nonumber\\
&\quad+ \int_0^s \int_0^a p(p-2)\|\nabla\Phi(\theta,z,u)\|^{p-4}_{L^2}\Big|\Big\langle  \nabla\Phi(\theta,z,u), -\mathrm{i} \sum_{j=1}^m z_j \nabla g_j(\Phi(\theta,z,u))\Big\rangle_{L^2}\Big|^2 \d \theta \d a\Big|\nonumber\\
&\leq p(p-1)8mK_{\tilde{g}}^2 |z|^2_{\R^m}\int_0^s \int_0^a \|\nabla\Phi(\theta,z,u)\|^{p}_{L^2}\d \theta \d a\nonumber\\
&\leq p(p-1)8mK_{\tilde{g}}^2 (1+M_3)  |z|^2_{\R^m} \|\nabla u\|^p_{L^2}. \label{lemma-proof-nabla-eq202}
 \end{align}
 Combining \eqref{lemma-proof-nabla-eq201} and \eqref{lemma-proof-nabla-eq202} shows the second inequality \eqref{lemma-est-nabla-Phi-eq2}.

\end{proof}

Since by Lemma \ref{lem-property-L_varepsilon} for every $\varepsilon$, the function $u\mapsto uL^{\varepsilon}(u)$ is global Lipschitz continuous, we may apply Banach's fixed point Theorem iteratively  and establish global existence of a mild solution in $L^2(\R^d)$ to the approximation equation  \eqref{app-SlogSE-marcus}. The proof is quite similar to that of \cite[Proposition 3.6]{BLZ} but simpler, and is therefore omitted here.

\begin{proposition}\label{prop-boundedness-L2} Let $0<\varepsilon<1$ and $0<T<\infty$ be fixed. Assume that $u_0\in L^p(\Omega;L^2(\R^d))$, $p\geq 2$. 
Under Assumption \ref{assum-main}, there exists a unique $\mathbb{F}$-adapted  $L^2(\R^d)$-valued c\`adl\`ag global mild solution $u_{\varepsilon}$ to the equation \eqref{app-SlogSE-marcus} such that $u_{\varepsilon} \in M^p_{\mathbb{F}}(Y_{T})$  and  it satisfies $\mathbb{P}$-a.s. for all $t\in[0,T]$
  \begin{align}
     u_{\varepsilon}(t)=&S_tu_0+\mathrm{i}\int_0^tS_{t-s}\big(2\lambda u_{\varepsilon}(s) L_{\varepsilon}(u_{\varepsilon}(s))\big)\d s\nonumber \\
     &+\int_0^t\int_BS_{t-s}\Big[\Phi(z,u_{\varepsilon}(s-))-u_{\varepsilon}(s-)\Big]\tilde{N}(\d s,\d z)\label{SEE-trucated-2}\\
     &+\int_0^t\int_BS_{t-s}\Big[\Phi(z,u_{\varepsilon}(s))-u_{\varepsilon}(s)+\mathrm{i}\sum_{j=1}^mz_jg_j(u_{\varepsilon}(s))\Big]\nu(\d z)\d s.\nonumber 
\end{align} 
 Moreover, we have $\mathbb{P}$-a.s. for every $t\in[0,T]$, 
\begin{align}\label{mass-conser-u-epsilon}
\|u_{\varepsilon}(t)\|_{L^2}=\|u(0)\|_{L^2}.
\end{align}

\end{proposition}

In Proposition \ref{prop-equv-two-solutions}, we will establish the equivalence of the mild formulation of the stochastic approximation equation \eqref{app-SlogSE-marcus} with the standard formulation through the It\^o process. The proof is postponed in Appendix.  By convention, we note $H^0(\R^d)=L^2(\R^d)$.

\begin{proposition}\label{prop-equv-two-solutions}
Let $u_0\in H^{m}(\R^d)$, $m=0,1$. The following are equivalent: 
\begin{enumerate}
\item[(i)]$u\in L(\Omega; L^\infty(0,T;H^m(\R^d)))$ is an It\^o process in $H^{m-2}(\R^d)$ such that $\mathbb{P}$-a.s. for all $t\in[0,T]$,
\begin{align}
 u(t)=&u_0+\int_0^t \mathrm{i} \left[\Delta u(s)+2\lambda u(s) L_{\varepsilon}(u(s))\right] \mathrm{d}s +\int_0^t\int_B[\Phi(z, u(s-))-u(s-)] \tilde{N}(\mathrm{d} s, \mathrm{d} z)\nonumber \\
& +\int_0^t\int_B\Big[\Phi(z, u(s))-u(s)+\mathrm{i} \sum_{j=1}^m z_j g_j(u(s))\Big] \nu(\mathrm{d} z) \mathrm{d} s;\label{strong-solution-eq}
\end{align}

\item[(ii)]  $u\in L(\Omega; L^\infty(0,T;H^m(\R^d)))$   satisfies  $\mathbb{P}$-a.s. for all $t\in[0,T]$,
  \begin{align}
     u(t)=&S_tu_0+\mathrm{i}\int_0^tS_{t-s}\big(2\lambda u(s) L_{\varepsilon}(u(s) )\big)\d s+\int_0^t\int_BS_{t-s}\Big[\Phi(z,u(s-))-u(s-)\Big]\tilde{N}(\d s,\d z)\nonumber\\
     &+\int_0^t\int_BS_{t-s}\Big[\Phi(z,u(s))-u(s)+\mathrm{i}\sum_{j=1}^mz_jg_j(u(s))\Big]\nu(\d z)\d s.\label{mild-solution-eq}
\end{align}
\end{enumerate}
\end{proposition}
\section{Uniform estimates}\label{sec-uni-est}
We will establish some uniform a priori estimates for the solution of \eqref{app-SlogSE-marcus} in $H^1(\R^d)$ under the additional assumption that $u_0\in L^p(\Omega;H^1(\R^d))$. 
\begin{proposition}Assume that $u_0\in L^p(\Omega;H^1(\R^d))$, $p\geq 2$. 
Under Assumption \ref{assum-main}, the solution $u_{\varepsilon}$ to \eqref{app-SlogSE-marcus} satisfies for $p\geq 2$,
\begin{align}\label{uni-bound-H1-eq}
\sup_{\varepsilon\in(0,1]}\mathbb{E}\sup _{t \in[0, T]}\left\|u_{\varepsilon}(t)\right\|_{H^1}^p\leq C(u_0,\lambda,m,T,p,K_{\tilde{g}})< \infty.
\end{align}
\end{proposition}
\begin{proof}
In view of Proposition \ref{prop-equv-two-solutions}, the mild solution in $L^2(\R^d)$ is equivalent to the following equation in $H^{-2}(\R^d)$, $\mathbb{P}$-a.s.
\begin{align}
 u_{\varepsilon}(t)=& u_0+  \int_0^t \mathrm{i}\left[\Delta u_{\varepsilon}(s)+2\lambda u_{\varepsilon}(s) L_{\varepsilon}(u_{\varepsilon}(s))\right] \mathrm{d} s+\int_0^t\int_B[\Phi(z, u_{\varepsilon}(s-))-u_{\varepsilon}(s-)] \tilde{N}(\mathrm{d} s, \mathrm{d} z)\nonumber \\
& +\int_0^t\int_B\Big[\Phi(z, u_{\varepsilon}(s))-u_{\varepsilon}(s)+\mathrm{i} \sum_{j=1}^m z_j g_j(u_{\varepsilon}(s))\Big] \nu(\mathrm{d} z) \mathrm{d} s\label{apprx-solution-strong-eq1}\\
=& u_0+ \int_0^t  \mathrm{i}\big[\Delta u_{\varepsilon}(s)+2\lambda u_{\varepsilon}(s)L_{\varepsilon}(u_{\varepsilon}(s))\big] \mathrm{d} s+\int_0^t\int_BG_{\varepsilon}(z,s-) \tilde{N}(\mathrm{d} s, \mathrm{d} z) \nonumber\\
& +\int_0^t\int_B H_{\varepsilon}(z,s) \nu(\mathrm{d} z) \mathrm{d} s,\label{app-SlogSE-strong-marcus}
\end{align}
where $G_{\varepsilon}(z,t) =\Phi(z, u_{\varepsilon}(t))-u_{\varepsilon}(t)$ and $H_{\varepsilon}(z,t)=\Phi(z, u_{\varepsilon}(t))-u_{\varepsilon}(t)+\mathrm{i} \sum_{j=1}^m z_j g_j(u_{\varepsilon}(t))$.

Applying the It\^o formula to the function $\mathcal{K}$ defined in \eqref{func-K-def} and the process \eqref{app-SlogSE-strong-marcus}, we get
\begin{align*}
   & \|\nabla u_{\varepsilon}(t)\|^p_{L^2} 
       =\|\nabla u_{\varepsilon}(0)\|^p_{L^2}+ \int_0^t \mathcal{K}^{\prime}[ u_\varepsilon(s)]( \mathrm{i}\Delta u_\varepsilon(s)+\mathrm{i}2\lambda u_{\varepsilon}(s) L_{\varepsilon}(u_{\varepsilon}(s)))\,\mathrm{d} s \\
       &+\int_0^t \int_B \| \nabla u_\varepsilon(s-)+\nabla G_\varepsilon(z, s-)\|^p_{L^2}-\| \nabla u_\varepsilon(s-)\|^p_{L^2} \,   \tilde{N}(\mathrm{d} s, \mathrm{d} z)\\
   &+\int_0^t \int_B\Big( \| \nabla u_\varepsilon(s)+\nabla G_\varepsilon(z, s)\|^p_{L^2}-\| \nabla u_\varepsilon(s)\|^p_{L^2} -  \mathcal{K}^{\prime}[ u_\varepsilon(s)](G_\varepsilon(z, s)) \Big)\, \nu(\mathrm{d} z)\d s \\
   &+ \int_0^t \int_B  \mathcal{K}^{\prime}[ u_\varepsilon(s)](H_\varepsilon(z, s))  \nu(dz)ds\\
       &=\|\nabla u_{\varepsilon}(0)\|^p_{L^2}+ \int_0^t \mathcal{K}^{\prime}[ u_\varepsilon(s)]( \mathrm{i}\Delta u_\varepsilon(s)+\mathrm{i} 2\lambda u_{\varepsilon} (s)L_{\varepsilon}(u_{\varepsilon}(s)))\mathrm{~d} s \\
       &+\int_0^t \int_B \| \nabla \Phi(z,  u_\varepsilon(s-))\|^p_{L^2}-\| \nabla u_\varepsilon(s-)\|^p_{L^2}\,    \tilde{N}(\mathrm{d} s, \mathrm{d} z)\\
   &+\int_0^t \int_B\Big(\| \nabla \Phi(z,  u_\varepsilon(s))\|^p_{L^2}-\| \nabla u_\varepsilon(s)\|^p_{L^2} \\
   &\quad\quad\quad\quad\quad+p\| \nabla u_\varepsilon(s)\|^{p-2}_{L^2} \big\langle \nabla u_\varepsilon(s),  \mathrm{i} \sum_{j=1}^m z_j \nabla g_j(u_{\varepsilon}(s)\big\rangle_{L^2} \Big) \nu(\mathrm{d} z)\d s \\
    &:=\|\nabla u_{\varepsilon}(0)\|^p_{L^2}+I_1(t)+I_2(t)+I_3(t),
\end{align*}
almost surely for all $t\in[0,T]$.  By the self-adjointness of $\Delta$,  the following identity holds for all $v \in H^2$:
\begin{align*}
\langle \nabla  v,\mathrm{i}\nabla \Delta v\rangle_{L^2}=\operatorname{Re}\int_{\R^d}\Delta v \overline{\mathrm{i}  \Delta v} \d x=0.
\end{align*}
 Note that
 \begin{align*}
 \nabla(L_{\varepsilon}(u_{\varepsilon}(s)))=\frac{(1-\varepsilon^2)|u_{\varepsilon}(s)|^{-1}\operatorname{Re}(\overline{u_{\varepsilon}(s)}\nabla(u_{\varepsilon}(s)))}{(\varepsilon+|u_{\varepsilon}(s)|)(1+\varepsilon|u_{\varepsilon}(s)|)}.
  \end{align*}
  Straightforward calculations yield
\begin{align*}
\big\langle \nabla u_\varepsilon(s), \nabla\big(\mathrm{i}\Delta u_\varepsilon(s)+\mathrm{i} 2\lambda  u_{\varepsilon}(s) L_{\varepsilon}(u_{\varepsilon}(s))\big)\big\rangle_{L_2}
&=\big\langle \nabla u_\varepsilon(s), \mathrm{i}2\lambda u_{\varepsilon}(s) \nabla(L_{\varepsilon}(u_{\varepsilon}(s))) \big\rangle_{L_2}\\
&\leq  2|\lambda| \|\nabla u_{\varepsilon}(s)\|^2_{L^2}.
 \end{align*}
By \eqref{func-K-first-der}, the first integral $I_1(t)$ can be estimated as
  \begin{align*}
  I_1(t)&=\int_0^t \mathcal{K}^{\prime}[ u_\varepsilon(s)]( \mathrm{i}\Delta u_\varepsilon(s)+\mathrm{i}2\lambda u_{\varepsilon}(s) L_{\varepsilon}(u_{\varepsilon}(s)))\,\mathrm{d} s \\
  &\lesssim_{\lambda }  \int_0^t \|\nabla u_\varepsilon(s) \|_{L^2}^p \d s.
\end{align*}
For the second term $I_2(t)$, using the Burkholder inequality, Lemma \ref{lem-est-nabla-Phi} and Young's inequality, we have
\begin{align*}
\EE\sup_{0\leq t\leq T}|I_2(t)|
& \leq  \EE \left( \int_0^T \int_B\Big| \| \nabla \Phi(z,  u_\varepsilon(s))\|^p_{L^2}-\| \nabla u_\varepsilon(s)\|^p_{L^2}\Big|^2  \nu(dz)ds\right)^{\frac12}\\
&\lesssim_{m,K_{\tilde{g}}} \EE \left( \int_0^T \|\nabla u_{\varepsilon}(s)\|^{2p}_{L^2}\left(\int_B |z|_{\R^m} ^2\nu(dz) \right)ds\right)^{\frac12}\\
&\lesssim  \EE \left( \int_0^T \|\nabla u_{\varepsilon}(s)\|^{2p}_{L^2}ds\right)^{\frac12}\\
&\lesssim \Big(\delta \EE\sup_{0\leq t\leq T}\|\nabla u_{\varepsilon}(t)\|^p +C(\delta) \int_0^T \EE \sup_{0\leq s\leq t}\|\nabla u_{\varepsilon}(s)\|^p dt\Big),
\end{align*}
where by assumption $\int_B|z|_{\mathbb{R}^m}^2 \nu(\mathrm{d} z)<\infty$. 
For the third term $I_3(t)$, by using Lemma \ref{lem-est-nabla-Phi}, we deduce
\begin{align*}
\EE\sup_{0\leq t\leq T}\big|I_3(t) \big|
&\lesssim_{m,K_{\tilde{g}}} \EE \int_0^T\int_B \|\nabla u_\varepsilon(s) \|_{L^2}^p|z|^2_{\R^m}  \nu(\mathrm{d} z)\d s\\
&\lesssim \EE \int_0^T \sup_{0\leq s\leq t} \|\nabla u_\varepsilon(s) \|_{L^2}^p \d t.
\end{align*}
It follows that
\begin{align*}
 \EE\sup_{0\leq t\leq T} \|\nabla u_{\varepsilon}(t)\|^p_{L^2} 
      & \lesssim_{m,p,\lambda,K_{\tilde{g}}} \EE\|\nabla u_{\varepsilon}(0)\|^p_{L^2}+\delta \EE\sup_{0\leq t\leq T}\|\nabla u_{\varepsilon}(t)\|^p \\
      &\quad\quad\quad +C(\delta)\int_0^T \EE \sup_{0\leq s\leq t}\|\nabla u_{\varepsilon}(s)\|^p dt + \int_0^T \EE \sup_{0\leq s\leq t}\|\nabla u_{\varepsilon}(s)\|^p dt.
\end{align*}
Taking $\delta$ sufficiently small yields
\begin{align*}
 \EE\sup_{0\leq t\leq T} \|\nabla u_{\varepsilon}(t)\|^p_{L^2} 
       \lesssim_{m,p,\lambda,K_{\tilde{g}}} &\,\EE\|\nabla u_{\varepsilon}(0)\|^p_{L^2}+ \int_0^T \EE \sup_{0\leq s\leq t}\|\nabla u_{\varepsilon}(s)\|^p \d t.
\end{align*}
Finally, applying the Gronwall Lemma yields
\begin{align*}
 \EE\sup_{0\leq t\leq T} \|\nabla u_{\varepsilon}(t)\|^p_{L^2} \leq C\EE\|\nabla u_{\varepsilon}(0)\|^p_{L^2}e^{CT},
 \end{align*}
 where the constant $C=C(u_0,\lambda,m,T,p,K_{\tilde{g}})>0$ is uniform in $\varepsilon\in(0,1]$.

\end{proof}

Next we establish the estimate of the entropy function.
\begin{lemma}\label{lem-est-F-function}
Assume that $u_0\in L^p(\Omega;W)$, $p\geq 2$. 
Under Assumption \ref{assum-main}, the solution $u_{\varepsilon}$ satisfies for all $T \in \mathbb{R}^+$ 
 \begin{equation}
\mathbb{E}\left[\sup _{t \in[0, T]}\Big| \int_{\mathbb{R}^d} \big|  F(|u^\varepsilon(t)|)\big| \d x\Big|^p\right]\leq C(u_0,\lambda,m,T,p,K_{\tilde{g}})<\infty.
\end{equation}

\end{lemma}
\begin{proof}
For $l>0$, let $F_{k}(l):=\int_0^l (L_{1/k}(v)+1)\d v$, where $L_{1/k}(v)=\log \left(\frac{v+1/k}{1+{v}/k}\right)$, $v>0$.
Then the function $F_k$ is continuously Fr\'echet-differentiable and for $u,h_1,h_2\in \mathbb{C}$,
\begin{align}
&F_{k}'[|u|^2](h)=(L_{1/k}(|u|^2)+1)2\operatorname{Re}(\bar{u}h).\label{proof-prop-F-est-eq1}
\end{align}

Next we will use the standard mollifiers. Let $\tilde{\vartheta}: \mathbb{R}^d \rightarrow \mathbb{R}$ be a smooth test function such that $0 \leq \tilde{\vartheta}(x) \leq 1, x \in \mathbb{R}^d$,   $\tilde{\vartheta}(x)=\tilde{\vartheta}(-x)$, $\operatorname{supp}(\tilde{\vartheta}) \subset B(2)$ and  $\tilde{\vartheta}(x)=1$ when $x \in B(1)$. Let $\vartheta=\frac{\tilde{\vartheta}}{\int_{\mathbb{R}^d} \tilde{\vartheta(x)} d x}$, then   $\int_{\mathbb{R}^d} \vartheta(x) d x=1$. For any $\delta>0$, let $\vartheta_{\delta}(x)=$ $\varepsilon^{-d} \vartheta(x / \delta)$ and for any locally integrable function $g: \mathbb{R}^d \rightarrow \mathbb{R}^d$ we define the mollified approximation $g_{\delta}$ as
\begin{align*}
g_{\delta}(x)=\vartheta_{\delta} * g(x), \quad x \in \mathbb{R}^d.
\end{align*}
Let $ u_{\varepsilon}$ be the solution of \eqref{app-SlogSE-marcus} in the integral form \eqref{apprx-solution-strong-eq1} in $H^{-2}(\R^d)$.  Apply convolution with the mollifiers $\vartheta_{\delta}$ to both sides of equation \eqref{apprx-solution-strong-eq1} to obtain for each $x \in \mathbb{R}^d$ 
 \begin{align}
 u_{\varepsilon,\delta}(t)(x)=&(u_0)_{\delta}+ \int_0^t \mathrm{i}\big[\Delta u_{\varepsilon,\delta}(s)(x)+2\lambda \big(u_{\varepsilon}(s) L_{\varepsilon}(u_{\varepsilon}(s))\big)_{\delta}(x)\big] \mathrm{d} s\nonumber\\
 &+\int_0^t\int_B[(\Phi(z, u_{\varepsilon}(s-)))_{\delta}(x)-u_{\varepsilon,\delta}(s-)(x)] \tilde{N}(\mathrm{d} s, \mathrm{d} z)\nonumber \\
& +\int_0^t\int_B\big[(\Phi(z, u_{\varepsilon}(s)))_{\delta}(x)-u_{\varepsilon,\delta}(s)(x)+\mathrm{i} \sum_{j=1}^m z_j (g_j(u_{\varepsilon}(s)))_{\delta}(x)\big] \nu(\mathrm{d} z) \mathrm{d} s,\label{prop-proof-molli-eq-1}
\end{align}
for all $t\in[0,T]$, on some full probability measure set $\Omega_x\in\mathcal{F}_0$ (independent of $t$) with $\mathbb{P}(\Omega_x)=1$, where we used the notation $u_{\varepsilon,\delta}=(u_{\varepsilon})_{\delta}$ for simplicity.  To remove the dependence on $x$ in the almost sure event, we need the continuity of all the terms of \eqref{prop-proof-molli-eq-1} in $x$. Let us denote 
\begin{align*}
& J_1(t,x):= u_{\varepsilon,\delta}(t)(x),\\
& J_2(t,x):=(u_0)_{\delta}(x)+ \int_0^t \mathrm{i}\big[\Delta u_{\varepsilon,\delta}(s)(x)+2\lambda \big(u_{\varepsilon}(s)L_{\varepsilon}(u_{\varepsilon}(s))\big)_{\delta}(x)\big] \mathrm{d} s,\\
 &J_3(t,x):=\int_0^t\int_B(\Phi(z, u_{\varepsilon}(s-))_{\delta}(x)-u_{\varepsilon,\delta}(s-)(x) \,\tilde{N}(\mathrm{d} s, \mathrm{d} z),\\
 &J_4(t,x):=\int_0^t\int_B(\Phi(z, u_{\varepsilon}(s)))_{\delta}(x)-u_{\varepsilon,\delta}(s)(x)+\mathrm{i} \sum_{j=1}^m z_j (g_j(u_{\varepsilon}(s)))_{\delta}(x) \nu(\mathrm{d} z) \mathrm{d} s.
\end{align*}
  By Lemma \ref{lem-stoch-lip-1} , we find $\mathbb{P}$-a.s.
  \begin{align*}
& \sup_{t\in[0,T]} \int_0^t\int_B\big| (\Phi(z, u_{\varepsilon}(s)))_{\delta}(x)-u_{\varepsilon,\delta}(s)(x)+\mathrm{i} \sum_{j=1}^m z_j (g_j(u_{\varepsilon}(s)))_{\delta}(x)\big| \nu(\mathrm{d} z) \mathrm{d} s\\
 & \lesssim_{m,K_{\tilde{g}}} \int_0^T\int_B \| \theta_{\delta}\|_{L^2}\| u_{\varepsilon}(s)\|_{L^2}|z |^2_{\mathbb{R}^m}\nu(\mathrm{d} z) \mathrm{d} s<\infty.
 \end{align*}
 Since the integrand  $x\mapsto (\Phi(z, u_{\varepsilon}(\cdot,\omega))_{\delta}(x)-u_{\varepsilon,\delta}(\cdot,\omega)(x)+\mathrm{i} \sum_{j=1}^m z_j (g_j(u_{\varepsilon}(\cdot,\omega)))_{\delta}(x)$ is continuous uniformly in $t\in[0,T]$ for all $\omega\in\Omega$, we infer $x\mapsto J_4(t,x)$ is continuous uniformly in $t\in[0,T]$  on some full probability measure set $\Omega^{1}\in\mathcal{F}_0$ with $\mathbb{P}(\Omega^1)=1$.  The continuity of $J_1(t,x)$ and $J_2(t,x)$ with respect to $x$ are also easy to verify.  So we may assume that functions $x\mapsto J_1(t,x)$ and $x\mapsto J_2(t,x)$ are continuous uniformly in $t\in[0,T]$ on $\Omega^1$.  Next we shall prove that $J_3$ has a continuous modification w.r.t. $x$.

  Let $S$ be a dense countable subset of $\mathbb{R}^d$. Let $\Omega^0=\big(\cap_{x\in S}\Omega_x\big)\cap \Omega^{1}$. Then we have $\mathbb{P}(\Omega^0)=1$.   
  Define 
\begin{align*}
\tilde{J}_3(t,x)(\omega):= J_1(t,x)(\omega)-J_1(0,x)(\omega)-J_2(t,x)(\omega)- J_4(t,x)(\omega),\quad \text{for }t\in[0,T],\;x\in\mathbb{R}^d,\;\omega \in \Omega^0
\end{align*}
and $\tilde{J}_3(t,x)(\omega)=0$ for $t\in[0,T],$ $x\in\mathbb{R}^d$, $\omega\in \Omega\backslash\Omega_0$.

 Given $x\in\mathbb{R}^d$, 
 there exists a sequence $\{x_n\}$ in $S$ such that $x_n\rightarrow x$.  In view of \eqref{prop-proof-molli-eq-1},   for every $\omega\in\Omega^0$, $t\in[0,T]$ and $x_n\in S$, we have
 \begin{align*}
  J_1(t,x_n)(\omega)=J_1(0,x_n)(\omega)+ J_2(t,x_n)(\omega)+ J_3(t,x_n)(\omega)+ J_4(t,x_n)(\omega).
 \end{align*}
and hence  $\tilde{J}_3(t,x_n)=J_3(t,x_n)$  for all $t\in[0,T]$ and $x_n\in S$ on $\Omega_0$.  
By using the Burkholder inequality and Lemma \ref{lem-G-H1}, we have
\begin{align*}
\EE\sup_{t\in[0,T]} \big| J_3(t,x) \big|^2
&=\EE\sup_{t\in[0,T]} \Big|\int_0^{t}\int_B(\Phi(z, u_{\varepsilon}(s-))_{\delta}(x)-u_{\varepsilon,\delta}(s-)(x)\, \tilde{N}(\mathrm{d} s, \mathrm{d} z)\Big|^2\\
&\lesssim \EE \int_0^{T} \int_B|(\Phi(z, u_{\varepsilon}(s))_{\delta}(x)-u_{\varepsilon,\delta}(s)(x)   |^2\nu(\d z)\d s\\
&\lesssim_{K_{\tilde{g}},m} \EE \int_0^{T} \int_B\|  u_{\varepsilon}(s)  \|_{L^2}^2 \|\vartheta_{\delta}\|_{L^2}^2 |z|_{\R^m}^2 \nu(\d z)\d s\\
&\lesssim\,T\EE \|  u_{0}  \|_{L^2}^2  \int_B|z|_{\R^m}^2\| \vartheta_{\delta}\|_{L^2}^2\nu(\d z)<\infty.
\end{align*}
Since the integrand $x\mapsto (\Phi(z, u_{\varepsilon}(s)))_{\delta}(x)-u_{\varepsilon,\delta}(s)(x)$ is continuous uniformly in $s\in[0,T]$, by the Lebesuge Dominated convergence theorem we infer
\begin{align*}
&\lim_{ x_n\rightarrow x}\EE\sup_{t\in[0,T]} \big| J_3(t,x_n)-J_3(t,x)\big|^2\\
&\lesssim \lim_{ x_n\rightarrow x}\EE \int_0^T \int_B \big| \big((\Phi(z, u_{\varepsilon}(s)))_{\delta}(x_n)-u_{\varepsilon,\delta}(s))(x_n)\big) -\big((\Phi(z, u_{\varepsilon}(s))_{\delta}(x)-u_{\varepsilon,\delta}(s)(x)\big)  \big|^2\,\nu(\d z)\d s\\
&=0.
\end{align*}
Then  there exists a subsequence $\{x_{n_k}\}$ of $\{x_n\}$ and a full measure set $\Omega'_x\subset \Omega^0$  (which may depend on $x$) 
such that  on  $\Omega'_x$, for all $t\in[0,T]$,
\begin{align*}
J_3(t,x)&=\lim_{ x_{n_k}\rightarrow x} J_3(t,x_{n_k})\\
&=\lim_{ x_{n_k}\rightarrow x} \big (J_1(t,x_{n_k})-J_1(0,x_{n_k})-J_2(t,x_{n_k}) -J_4(t,x_{n_k})\big)\\
&=J_1(t,x)-J_1(0,x_{n_k})-J_2(t,x) -J_4(t,x)\\
&=\tilde{J}_3(t,x),
\end{align*}
where we also used the continuity of $x\mapsto J_1(t,x), x\mapsto J_2(t,x)$ and $x\mapsto J_4(t,x)$ uniformly in $t\in[0,T]$.  
In other words, $\mathbb{P}\big(\tilde{J}_3(t,x)= J_3(t,x),\,\text{for all }t\in[0,T]\big)=1$, for all $x\in \mathbb{R}^d$.  From now on, when we consider the stochastic term $J_3$, we  only focus on its continuous modification $\tilde{J}_3$ of $J_3$. In this way, \eqref{prop-proof-molli-eq-1} holds for all $x$, $\mathbb{P}$-a.s.

Next we shall apply the  It\^o formula to the function $F_{k}(|u|^2)$ and then  integrate the result over  $\mathbb{R}^d$ to obtain
\begin{align}
& \int_{\R^d} F_k(|u_{\varepsilon, \delta}(t)|^2) \d x\label{proof-prop-F-est-eq3} \\
= & \int_{\R^d} F_k(|(u_0)_{\delta}|^2) d x+\int_0^t\int_{\R^d} F_{k}'[| u_{\varepsilon,\delta}(s)|^2]\big(\operatorname{i}[\Delta u_{\varepsilon,\delta}(s)+2\lambda \big(u_{\varepsilon}(s) L_{\varepsilon}(u_{\varepsilon}(s))\big)_{\delta}]\big)\d x\,\d s\nonumber\\
&+\int_0^t\int_B \int_{\R^d} F_k\big(|(\Phi(z, u_{\varepsilon}(s-)))_{\delta}|^2\big)-F_k\big(|u_{\varepsilon, \delta}(s-)|^2\big) \d x\, \tilde{N}(\mathrm{d} s, \mathrm{d} z)\nonumber\\
&+\int_0^t\int_B \int_{\R^d}  F_k\big(|(\Phi(z, u_{\varepsilon}(s))_{\delta}|^2\big)-F_k(|u_{\varepsilon, \delta}(s)|^2)\nonumber \\
&\quad\quad\quad\quad\quad\quad -2( L_{1 / k}(|u_{\varepsilon, \delta}(s)|^2)+1)\operatorname{Re}\left(\overline{u_{\varepsilon, \delta}(s)} \big((\Phi(z, u_{\varepsilon}(s)))_{\delta}-u_{\varepsilon,\delta}(s)) \big) \right) \d x\, \nu(\d z)\d s\nonumber\\
&+\int_0^t\int_B \int_{\R^d} ( L_{1 / k}(|u_{\varepsilon, \delta}(s)|^2)+1)\\
&\quad\quad\quad\quad\quad\quad\cdot 2\operatorname{Re}\Big(\overline{u_{\varepsilon, \delta}(s)} \big((\Phi(z, u_{\varepsilon}(s)))_{\delta}-u_{\varepsilon,\delta}(s)+\mathrm{i} \sum_{j=1}^m z_j (g_j(u_{\varepsilon}(s))_{\delta}\big)\Big) \d x\, \nu(\d z)\d s,\nonumber
\end{align}
where we also used stochastic Fubini's theorem (c.f. \cite{Zhu+Liu-21}). 
By using \eqref{proof-prop-F-est-eq1} and subsequently applying the integration by parts formula, we have
\begin{align}
\int_{\R^d} F_{k}'[| u_{\varepsilon,\delta}(s)|^2](\operatorname{i}\Delta u_{\varepsilon,\delta}(s))\d x& =\int_{\R^d}  (L_{1/k}(| u_{\varepsilon,\delta}(s)|^2)+1)2\operatorname{Re}(\overline{ u_{\varepsilon,\delta}(s)} \operatorname{i}\Delta u_{\varepsilon,\delta}(s)  )   \d x  \nonumber\\
&=-\int_{\R^d}  (L_{1/k}(| u_{\varepsilon,\delta}(s)|^2)+1)2\operatorname{Re}(\overline{ \nabla u_{\varepsilon,\delta}(s)}\operatorname{i}\nabla  u_{\varepsilon,\delta}(s) )  \d x\nonumber\\
&\quad-\int_{\R^d} f_k(| u_{\varepsilon,\delta}(s)|^2)2\operatorname{Re}(\overline{ u_{\varepsilon,\delta}(s)}\operatorname{i}\nabla  u_{\varepsilon,\delta}(s) )2\operatorname{Re}(\overline{ u_{\varepsilon,\delta}(s)}\nabla  u_{\varepsilon,\delta}(s) ) \d x \nonumber\\
&=    4  \int_{\R^d} f_k(|u_{\varepsilon,\delta}(s)|^2) \operatorname{Im}(\overline{u_{\varepsilon,\delta}(s)} \nabla u_{\varepsilon,\delta}(s)) \operatorname{Re}(\overline{u_{\varepsilon,\delta}(s)} \nabla u_{\varepsilon,\delta}(s)) \d x, \label{proof-prop-F-est-eq2}
\end{align}
where $f_k\left(|u|^2\right):=2(1-k^{-2})(k^{-1}+|u|^2)^{-1}(1+k^{-1}|u|^2)^{-1}$ and we used  the fact that $\operatorname{Re}(\overline{u}(\text{i}v))=-\operatorname{Im}(\bar{u}v)$ and $\operatorname{Re}(\overline{u}(\text{i}u))=0$. 
By substituting \eqref{proof-prop-F-est-eq2} into \eqref{proof-prop-F-est-eq3} and making some cancellations, we obtain
\begin{align*}
& \int_{\R^d} F_k(|u_{\varepsilon, \delta}(t)|^2) \d x \\
=& \int_{\R^d} F_k(|(u_0)_\delta|^2) d x+4 \int_0^t \int_{\R^d} f_k(|u_{\varepsilon, \delta}(s)|^2) \operatorname{Re}\left(\overline{u_{\varepsilon, \delta}(s)} \nabla u_{\varepsilon, \delta}(s)\right) \operatorname{Im}\left(\overline{u_{\varepsilon, \delta}(s)} \nabla u_{\varepsilon, \delta}(s)\right) \d x\, \d s\\
&-4 \lambda \operatorname{Im} \int_0^t \int_{\R^d}\big(L_{1 / k}(|u_{\varepsilon, \delta}(s)|^2)+1\big) \overline{u_{\varepsilon, \delta}(s)}\left(u_{\varepsilon}(s) L_{\varepsilon}\left(u_{\varepsilon}(s)\right)\right)_\delta \d x \,\d s\\
&+\int_0^t\int_B \int_{\R^d} F_k\big(|(\Phi(z, u_{\varepsilon}(s-))_{\delta}|^2\big)-F_k\big(|u_{\varepsilon, \delta}(s-)|^2\big) \d x\, \tilde{N}(\mathrm{d} s, \mathrm{d} z)\\
&+\int_0^t\int_B \int_{\R^d}  F_k\big(|(\Phi(z, u_{\varepsilon}(s)))_{\delta}|^2\big)-F_k\big(|u_{\varepsilon, \delta}(s)|^2\big) \\
&\quad\quad\quad+( L_{1 / k}(|u_{\varepsilon, \delta}(s)|^2)+1)2\operatorname{Re}\Big(\overline{u_{\varepsilon, \delta}(s)} \big(\mathrm{i} \sum_{j=1}^m z_j (g_j(u_{\varepsilon}(s)))_{\delta}\big)\Big) \d x\, \nu(\d z)\d s.
\end{align*}
By using the inequality $|L_{1/k}(v)| \leq \log k$, for all $v>0$, the mean value theorem and Lemma \ref{lem-stoch-lip-1}, we find $\mathbb{P}$-a.s. for all $s\in[0,T]$
\begin{align*}
&\Big|\int_{\R^d}  F_k\big(|(\Phi(z, u_{\varepsilon}(s)))_{\delta}|^2\big)-F_k\big(|u_{\varepsilon, \delta}(s)|^2\big) \d x\Big|\\
&\leq \Big|\int_{\R^d} \int_{|u_{\varepsilon,\delta}(s) |^2}^{ |(\Phi(z, u_{\varepsilon,\delta}(s)))_{\delta}|^2}(L_{1/k}(v)+1)\d v\,\d x\Big|\\
&\leq \int_{\R^d} (\log k+1)\big| |(\Phi(z, u_{\varepsilon}(s)))_{\delta}|^2-   |u_{\varepsilon,\delta}(s) |^2 \big| \d x\\
&\lesssim(\log k+1) \int_{\R^d}  |\operatorname{Re}\big[ \big(\theta (\Phi(z, u_{\varepsilon}(s)))_{\delta}+(1-\theta)u_{\varepsilon,\delta}(s) \big)\overline {( \Phi(z, u_{\varepsilon}(s)))_{\delta} - u_{\varepsilon,\delta}(s)} \big]\big|\d x\\
&\lesssim (\log k+1)\big( \| (\Phi(z, u_{\varepsilon}(s)))_{\delta}\|_{L^2}+  \|u_{\varepsilon,\delta}(s) \|_{L^2}\big)\|( \Phi(z, u_{\varepsilon}(s)))_{\delta} - u_{\varepsilon,\delta}(s)\|_{L^2}\\
&\lesssim_{m,K_{\tilde{g}}} (\log k+1)\big( \| \Phi(z, u_{\varepsilon}(s))\|_{L^2}+  \|u_{\varepsilon}(s) \|_{L^2}\big)|z|_{\R^m}\|u_{\varepsilon,\delta}(s) \|_{L^2}\\
&\lesssim(\log k+1) \|u_{\varepsilon}(s) \|_{L^2}^2|z|_{\R^m}\\
&=(\log k+1) \|u_{0} \|_{L^2}^2|z|_{\R^m},
\end{align*}
where we also used the fact that $| \Phi(z, y)|=|y|$ for all $z\in B$ and $y\in\mathbb{C}$ (see  \cite[Lemma 3.2]{BLZ}). 
Similarly, we have
\begin{align*}
&\Big|\int_{\R^d}  F_k\big(|(\Phi(z, u_{\varepsilon}(s)))_{\delta}|^2\big)-F_k\big(|u_{\varepsilon, \delta}(s)|^2\big) +( L_{1 / k}(|u_{\varepsilon, \delta}(s)|^2)+1)2\operatorname{Re}\Big(\overline{u_{\varepsilon, \delta}(s)} \big(\mathrm{i} \sum_{j=1}^m z_j (g_j(u_{\varepsilon}(s)))_{\delta}\big)\Big) \d x\Big|\\
&\lesssim_{m,K_{\tilde{g}}}(\log k+1) \|u_{0} \|_{L^2}^2|z|_{\R^m}^2.
\end{align*}
Since $g_{\delta}\rightarrow g$ in $L^2(\R^d)$, for any $g\in L^2(\R^d)$,  by using the above observations, the uniform boundedness \eqref{uni-bound-H1-eq} of $u_{\varepsilon}$, the It\^o isometry property of the stochastic integral, and the Lebesgue Dominated Convergence Theorem, we get as $\delta\rightarrow 0$ (passing to a subsequence if necessary)
\begin{align*}
&\int_0^t\int_B \int_{\R^d} F_k\big(|(\Phi(z, u_{\varepsilon}(s-)))_{\delta}|^2\big)-F_k\big(|u_{\varepsilon, \delta}(s-)|^2\big) \d x\, \tilde{N}(\mathrm{d} s, \mathrm{d} z)\\
&\longrightarrow \int_0^t\int_B \int_{\R^d} F_k\big(|\Phi(z, u_{\varepsilon}(s-))|^2\big)-F_k\big(|u_{\varepsilon}(s-)|^2\big) \d x\, \tilde{N}(\mathrm{d} s, \mathrm{d} z)\quad \mathbb{P}\text{-a.s.}\\
&\int_0^t\int_B \int_{\R^d}  F_k\big(|(\Phi(z, u_{\varepsilon}(s)))_{\delta}|^2\big)-F_k\big(|u_{\varepsilon, \delta}(s)|^2\big)\\
&\quad\quad\quad +( L_{1 / k}(|u_{\varepsilon, \delta}(s)|^2)+1)2\operatorname{Re}\Big(\overline{u_{\varepsilon, \delta}(s)} \big(\mathrm{i} \sum_{j=1}^m z_j (g_j(u_{\varepsilon}(s)))_{\delta}\big)\Big) \d x\, \nu(\d z)\d s\\
&\longrightarrow \int_0^t\int_B \int_{\R^d}  F_k\big(|\Phi(z, u_{\varepsilon}(s))|^2\big)-F_k\big(|u_{\varepsilon}(s)|^2\big) \\
&\quad\quad\quad+( L_{1 / k}(|u_{\varepsilon}(s)|^2)+1)2\operatorname{Re}\Big(\overline{u_{\varepsilon}(s)} \big(\mathrm{i} \sum_{j=1}^m z_j g_j(u_{\varepsilon}(s))\big)\Big) \d x\, \nu(\d z)\d s\;\mathbb{P}\text{-a.s.}
\end{align*}
Therefore, using these limiting results, we obtain, $\mathbb{P}$-a.s.
\begin{align*}
 & \int_{\R^d} F_k(|u_{\varepsilon}(t)|^2) \d x \\
=& \int_{\R^d} F_k(|u_0|^2) \d x+4 \int_0^t \int_{\R^d} f_k(|u_{\varepsilon}(s)|^2) \operatorname{Re}\big(\overline{u_{\varepsilon}(s)} \nabla u_{\varepsilon}(s)\big) \operatorname{Im}\big(\overline{u_{\varepsilon}(s)} \nabla u_{\varepsilon}(s)\big) \,\d x\d s\\
&-4 \lambda \operatorname{Im} \int_0^t \int_{\R^d}\big(L_{1 / k}(|u_{\varepsilon}(s)|^2)+1\big) \overline{u_{\varepsilon}(s)}(u_{\varepsilon}(s) L_{\varepsilon}(u_{\varepsilon}(s)))\, \d x \d s\\
&+\int_0^t\int_B \int_{\R^d} F_k\big(|\Phi(z, u_{\varepsilon}(s-)|^2\big)-F_k\big(|u_{\varepsilon}(s-)|^2\big) \d x\, \tilde{N}(\mathrm{d} s, \mathrm{d} z)\\
&+\int_0^t\int_B \int_{\R^d}  F_k\big(|\Phi(z, u_{\varepsilon}(s))|^2\big)-F_k\big(|u_{\varepsilon}(s)|^2\big) \\
&\quad\quad\quad+( L_{1 / k}(|u_{\varepsilon}(s)|^2)+1)2\operatorname{Re}\Big(\overline{u_{\varepsilon}(s)} \big(\mathrm{i} \sum_{j=1}^m z_j (g_j(u_{\varepsilon}(s)))\big)\Big) \d x\, \nu(\d z)\d s.
\end{align*}
Recall that $g_j(u)=\tilde{g}_j(|u|^2)u$. Since $\operatorname{Im}(\overline{v}v)= \operatorname{Re}(\overline{v}\text{i}v)=0$ and $|\Phi(z, y)|= |y|$, for all $z\in B$ and $y\in\mathbb{C}$ (see  \cite[Lemma 3.2]{BLZ}), the above formula can be simplified to be $\mathbb{P}$-a.s. for all $t\in[0,T]$,
\begin{align*}
 \int_{\R^d} F_k(|u_{\varepsilon}(t)|^2) \d x =  \int_{\R^d} F_k(|u_0|^2) \d x+4 \int_0^t \int_{\R^d} f_k(|u_{\varepsilon}(s)|^2) \operatorname{Re}(\overline{u_{\varepsilon}(s)} \nabla u_{\varepsilon}(s)) \operatorname{Im}(\overline{u_{\varepsilon}(s)} \nabla u_{\varepsilon}(s)) \d x\, \d s.
\end{align*}
Observe that for $u>0$,
\begin{align*}
F_k(u)=u L_{1 / k}(u)+u-\left(1-1/k^2\right) \int_0^u \frac{s}{(1/k+s)(1+s/k)} \d s.
\end{align*}
Since $ u_0\in L^p(\Omega;W)$,   we infer
\begin{align*}
\EE \sup_{k\in \mathbb{N}} \Big| \int_{\mathbb{R}^d}|F_k(|u_0|^2)| \d x\Big|^p &\leq \EE \sup_{k\in \mathbb{N}} \Big| \int_{\mathbb{R}^d} \big| | u_0|^2 L_{1 / k}(|u_0|^2)\big|+2| u_0|^2 \d x\Big|^p\\
&\leq \EE  \Big| \int_{\mathbb{R}^d}  \big| | u_0|^2 \log |u_0|^2\big|+2| u_0|^2 \d x \Big|^p <\infty .
\end{align*}
Since $F_k(|y|^2) \rightarrow|y|^2 \log|y|^2$ for any $y\in \R$, as $m \rightarrow \infty$,  it follows from the Lebesgue dominated convergence theorem that as $k\rightarrow \infty$, 
\begin{align*}
\int_{\R^d} F_k(|u_0|^2) \d x \rightarrow \int_{\R^d} |u_0|^2 \log|u_0|^2 \d x.
\end{align*}
Using the fact that $f_k(|u_{\varepsilon}|^2) \leq 2\left|u_{\varepsilon}\right|^{-2}$ and the uniform boundedness \eqref{uni-bound-H1-eq} of $H^1$-norm, we derive that
\begin{align*}
&\EE\sup_{t\in[0,T]} \Big| \int_{\R^d} F_k(|u_{\varepsilon}(t)|^2) \d x\Big|^p\\
&\lesssim_p \Big|  \int_{\R^d} F_k(|u_0|^2) \d x \Big| ^p+ \EE\sup_{t\in[0,T]} \Big|  \int_0^t \int_{\R^d} f_k(|u_{\varepsilon}(s)|^2) \operatorname{Re}\big(\overline{u_{\varepsilon}(s)} \nabla u_{\varepsilon}(s)\big) \operatorname{Im}\big(\overline{u_{\varepsilon}(s)} \nabla u_{\varepsilon}(s)\big) \d x\, \d s \Big| ^p\\
&\lesssim \Big|  \int_{\R^d}   \big (\big| | u_0|^2 \log|u_0|^2\big|+2| u_0|^2\big) \d x \Big| ^p+\EE\sup_{t\in[0,T]} \Big( \int_0^t \|u_{\varepsilon}(s) \|_{H^1}^2 \d s\Big)^{p}\\
&\leq C(u_0,\lambda,m,T,p,K_{\tilde{g}})<\infty,
\end{align*}
where the constant is independent of $k$ and $\varepsilon$. 
Hence, by Fatou's lemma, we obtain
\begin{align*}
\EE\sup_{t\in[0,T]}\Big| \int_{\R^d} F(|u_{\varepsilon}(t)|) \d x\Big|^p \leq C(u_0,\lambda,m,T,p,K_{\tilde{g}}).
\end{align*}
Using the fact that $ v^2 \log v^2 \leq C(\delta)(v^2+v^{2+\delta})$ for $v>1$, we derive
\begin{align*}
& \EE\sup_{t\in[0,T]}\Big|\int_{\mathbb{R}^d} \big| | u_{\varepsilon}(t)|^2 \log |u_{\varepsilon}(t)|^2\big| \d x\Big|^p \\
&=\EE\sup_{t\in[0,T]}\Big| -\int_{\mathbb{R}^d}  | u_{\varepsilon}(t)|^2 \log |u_{\varepsilon}(t)|^2\d x+2 \int_{\{|u_{\varepsilon}| > 1\}}| u_{\varepsilon}(t)|^2 \log |u_{\varepsilon}(t)|^2 \d x\Big|^p\\
&\lesssim_p \EE\sup_{t\in[0,T]}\Big| \int_{\mathbb{R}^d} | u_{\varepsilon}(t)|^2 \log|u_{\varepsilon}(t)|^2\d x\Big|^p  + C( \delta)\EE\sup_{t\in[0,T]}\Big(\|u_{\varepsilon}(t)\|_{L^2}^{2 p}+\|u_{\varepsilon}(t)\|_{L^{2+\delta}}^{(2+\delta) p}\Big)\\
&\leq C(u_0,\lambda,m,T,p,K_{\tilde{g}})+ C( \delta)\EE\sup_{t\in[0,T]}\Big(\|u_{\varepsilon}(t)\|_{L^2}^{2 p}+\|u_{\varepsilon}(t)\|_{H^1}^{(2+\delta) p}\Big)\\
&\leq C(u_0,\lambda,m,T,p,K_{\tilde{g}}).
\end{align*}
where we used the embedding $H^1(\R^d) \hookrightarrow  L^{2+\delta}(\R^d) $, for $0<\delta<\frac{4}{d-2}$.



\end{proof}


\section{Proof of Theorem \ref{Them-main}}\label{sec-proof-mian}

 The primary approach to the proof of Theorem \ref{Them-main} is structured as follows.  
Using the above uniform estimates of the approximate solutions, we can deduce the existence of a limit function for a subsequence of the approximate solutions in the weak (or weak$^*$) topology. To verify that the limit function satisfies the original equation, we will strengthen the convergence of the approximate solutions by  proving that the solutions of the approximation equation \eqref{intro-app-SlogSE} form a Cauchy sequence in the Banach space $L^2(\Omega ; L^{\infty}([0, T] ; L^2_{\text{loc}}(\R^d)))$. Using the uniformly estimate of the  entropy function in the Orlicz space we obtain the limit solution is in the energy space $W$. Our proof is independence of the conservation of the energy and the sign of $\lambda$. 


\begin{proof} The proof is divided into six steps. 

\textbf{Step 1 } We claim that $\left\{u_{\varepsilon}\right\}_{ 0<\varepsilon<1}$ forms a Cauchy sequence in $L^2(\Omega ; L^{\infty}(0, T ; L^2_{\text{loc}}(\R^d)))$.

Let ${\zeta} \in C_c^{\infty}\left(\mathbb{R}^d\right)$ be a mollifier 
 satisfying
\begin{align*}
{\zeta}(x)=\left\{\begin{array}{ll}
1 & \text { if }|x| \leq 1, \\
0 & \text { if }|x| \geq 2,
\end{array} \quad 0 \leq {\zeta}(x) \leq 1 \quad \text { for all } x \in \mathbb{R}^d .\right.
\end{align*}
 For $R>0$ we set $\zeta_R:=\zeta(x / R)$. 
Let $\varepsilon, \mu \in(0,1)$. Let $u_{\varepsilon}$ and $u_{\mu}$ be solutions of It\^o formulation to \eqref{app-SlogSE-marcus} with respect to $\varepsilon$ and $\mu$ respectively. 
 By subtracting the equation for $u_{\varepsilon}$ from the equation  for $u_{\mu}$, we obtain in $H^{-2}(\R^d)$
\begin{align*}
& u_{\varepsilon}(t)-u_{\mu} (t)
=\int_0^t \mathrm{i} \Delta\big(u_{\varepsilon}(s)-u_{\mu}(s)\big) \d s+\mathrm{i}2 \lambda\big(u_{\varepsilon}(s) L_{\varepsilon}(u_{\varepsilon}(s)) - u_{\mu}(s) L_{\mu}(u_{\mu}(s))\big) \d s\\
&+\int_0^t\int_B\big[\Phi(z, u_{\varepsilon}(s-))-\Phi(z, u_{\mu}(s-))-\big(u_{\varepsilon}(s-)-u_{\mu}(s-)\big)\big] \tilde{N}(\mathrm{d} s, \mathrm{d} z) \\
& +\int_0^t\int_B\Big[\Phi(z, u_{\varepsilon}(s))-\Phi(z, u_{\mu}(s))-\big(u_{\varepsilon}(s)-u_{\mu}(s)\big)+\mathrm{i} \sum_{j=1}^m z_j \big(g_j(u_{\varepsilon}(t))- g_j(u_{\mu}(t))\big)\Big] \nu(\mathrm{d} z) \mathrm{d} s.
\end{align*}
Applying the It\^o formula to $\left\|\zeta_R(u_{\varepsilon}(t)-u_{\mu}(t))\right\|_{L^2}^2$, we have
\begin{align*}
& \|\zeta_R(u_{\varepsilon}(t)-u_{\mu}(t))\|_{L^2}^2 =\int_0^t2 \langle \mathrm{i} \nabla(\zeta_R^2) \nabla\left(u_{\varepsilon}(s)-u_\mu(s)\right), u_{\varepsilon}(s)-u_\mu(s)\rangle_{L^2}\d s\\
&+\int_0^t 2\left\langle   \mathrm{i}2 \lambda\zeta_R^2\big(u_{\varepsilon}(s) L_{\varepsilon}(u_{\varepsilon}(s)) - u_{\mu}(s) L_{\mu}(u_{\mu}(s))\big),u_{\varepsilon}(s)-u_{\mu}(s) \right\rangle_{L^2} d s\\
&+\int_0^t\int_B \big\| \zeta_R \big(\Phi(z, u_{\varepsilon}(s-)) - \Phi(z, u_{\mu}(s-))\big)\big\|_{L^2}^2- \big\| \zeta_R\big(u_{\varepsilon}(s-)-u_{\mu}(s-)\big)\big\|_{L^2}^2   \tilde{N}(\mathrm{d} s, \mathrm{d} z) \\
&+\int_0^t\int_B \big\| \zeta_R\big(\Phi(z, u_{\varepsilon}(s)) - \Phi(z, u_{\mu}(s))\big)\big\|_{L^2}^2- \big\| \zeta_R\big(u_{\varepsilon}(s)-u_{\mu}(s)\big)\big\|_{L^2}^2 \\
&\quad \quad\quad-2 \langle u_{\varepsilon}(s)-u_{\mu}(s),\zeta_R^2\big(\Phi(z, u_{\varepsilon}(s))-\Phi(z, u_{\mu}(s))-\big(u_{\varepsilon}(s)-u_{\mu}(s)\big)\big)\rangle_{L^2}  \nu(\d z)\d s\\
&+\int_0^t\int_B 2\big\langle u_{\varepsilon}(s)-u_{\mu}(s),\zeta_R^2\big(\Phi(z, u_{\varepsilon}(s))-\Phi(z, u_{\mu}(s))-\big(u_{\varepsilon}(s)-u_{\mu}(s)\big)\big)\rangle_{L^2}\\
&\quad \quad\quad+2\big\langle u_{\varepsilon}(s)-u_{\mu}(s), \mathrm{i} \sum_{j=1}^m z_j \zeta_R^2 \big(g_j(u_{\varepsilon}(t))- g_j(u_{\mu}(t))\big) \big\rangle_{L^2} \nu(\d z)\d s.
\end{align*}
Taking inequality \eqref{L-var-est-2} into account, we infer for $\delta\in(0,1)$,
\begin{align*}
&\big\langle \mathrm{i}2 \lambda \zeta_R^2\big(u_{\varepsilon}(s) L_{\varepsilon}(u_{\varepsilon}(s)) - u_{\mu}(s) L_{\mu}(u_{\mu}(s))\big),u_{\varepsilon}(s)-u_{\mu}(s) \big\rangle_{L^2}\\
&=-2\lambda \int_{\R^d} \operatorname{Im} \zeta_R^2\big(u_{\varepsilon}(s) L_{\varepsilon}(u_{\varepsilon}(s)) - u_{\mu}(s) L_{\mu}(u_{\mu}(s))\big)\overline{ u_{\varepsilon}(s)-u_{\mu}(s) } \d x\\
&\lesssim (1-\varepsilon^2)|\lambda|\int_{\R^d} |\zeta_R(u_{\varepsilon}(s)-u_{\mu}(s))|^2\d x+|\lambda|  |\varepsilon-\mu|\int_{\R^d} |\zeta_R^2(u_{\varepsilon}(s)-u_{\mu}(s))| \d x\\
&\quad+C(\delta)|\lambda|  |\varepsilon-\mu|^{\delta}\int_{\R^d} |\zeta_R^2(u_{\varepsilon}(s)-u_{\mu}(s))||u_{\mu}(s) |^{1+\delta} \d x\\
&\leq (1-\varepsilon^2)|\lambda|  \|\zeta_R(u_{\varepsilon}-u_{\mu} )\|_{L^2}^2+|\lambda| |\varepsilon-\mu|  \|\zeta_R^2( u_{\varepsilon}-u_{\mu})\|_{L_1}\\
&\quad+C(\delta)|\lambda|  |\varepsilon-\mu|^{\delta}\|\zeta_R^2(u_{\varepsilon}(s)-u_{\mu}(s))\|_{L^2}\|u_{\mu}(s) \|^{1+\delta}_{L^{2+2\delta}}.
\end{align*}
Substituting the above estimate yields
\begin{align}
& \|\zeta_R(u_{\varepsilon}(t)-u_{\mu}(t))\|^2\nonumber\\
  &\lesssim \frac{1}{R}\int_0^t\left\|\nabla\left(u_{\varepsilon}(s)-u_{\mu}(s)\right)\right\|_{L^2}\left\|u_{\varepsilon}(s)-u_{\mu}(s)\right\|_{L^2}\d s+ (1-\varepsilon^2)\int_0^t |\lambda|  \|\zeta_R(u_{\varepsilon}(s)-u_{\mu} (s))\|_{L^2}^2\d s\nonumber\\
&+|\lambda| |\varepsilon-\mu| \int_0^t  \|\zeta_R^2( u_{\varepsilon}(s)-u_{\mu}(s)\|_{L_1}\d s\nonumber\\
&+C(\delta)|\lambda|  |\varepsilon-\mu|^{\delta}\int_0^t \|\zeta_R^2(u_{\varepsilon}(s)-u_{\mu}(s))\|_{L^2}\|u_{\mu}(s) \|^{1+\delta}_{L^{2+2\delta}}\d s\nonumber\\
&+\Big| \int_0^t\int_B \big\| \zeta_R \big(\Phi(z, u_{\varepsilon}(s-)) - \Phi(z, u_{\mu}(s-))\big)\big\|_{L^2}^2- \big\| \zeta_R\big(u_{\varepsilon}(s-)-u_{\mu}(s-)\big)\big\|_{L^2}^2   \tilde{N}(\mathrm{d} s, \mathrm{d} z)\Big|\nonumber\\
&+\Big| \int_0^t\int_B  \big\| \zeta_R \big(\Phi(z, u_{\varepsilon}(s)) - \Phi(z, u_{\mu}(s))\big)\big\|_{L^2}^2- \big\| \zeta_R\big(u_{\varepsilon}(s)-u_{\mu}(s)\big)\big\|_{L^2}^2 \nonumber\\
&\quad \quad\quad+2\big\langle u_{\varepsilon}(s)-u_{\mu}(s), \mathrm{i} \sum_{j=1}^m z_j \zeta_R^2\big(g_j(u_{\varepsilon}(t))- g_j(u_{\mu}(t))\big) \big\rangle_{L^2} \nu(\d z)\d s\Big| .\label{proof-mian-est-cau-eq11}
\end{align}
Next we will deal with each of the last two terms on the right hand side of \eqref{proof-mian-est-cau-eq11}  individually.
Observe that 
\begin{align*}
&\left|\Phi(s, z, y_1)-\Phi(s, z, y_2)\right|^2 - |y_1-y_2|^2\\
&\lesssim \big(|\Phi(s, z, y_1)-\Phi(s, z, y_2) |+|y_1-y_2|\big) \big| | \Phi(s, z, y_1)-\Phi(s, z, y_2)|-|y_1-y_2| \big|\\
&\lesssim_m |y_1-y_2|^2|z|_{\mathbb{R}^m},
\end{align*}
where we used \eqref{Phi-est-Lip} and \eqref{G-est-Lip-1} in the last inequality. 
Employing the Burkholder inequality and the above estimate, we derive that
\begin{align}
&\EE\sup_{t\in[0,T]}\Big| \int_0^t\int_B \big\| \zeta_R\big( \Phi(z, u_{\varepsilon}(s-)) - \Phi(z, u_{\mu}(s-)\big)\big\|_{L^2}^2- \big\|\zeta_R\big( u_{\varepsilon}(s-)-u_{\mu}(s-)\big)\big\|_{L^2}^2   \tilde{N}(\mathrm{d} s, \mathrm{d} z)\Big| \nonumber \\
&\leq \EE\Big(\int_0^T\int_B  \Big|\big\| \zeta_R\big(\Phi(z, u_{\varepsilon}(s)) - \Phi(z, u_{\mu}(s)\big)\big\|_{L^2}^2- \big\|\zeta_R\big( u_{\varepsilon}(s)-u_{\mu}(s))\big\|_{L^2}^2\Big|^2\;\nu(\d z)\d s\Big)^{\frac12}\nonumber\\
&\lesssim_m  \EE\Big(\int_0^T\int_B  \big\|\zeta_R\big( u_{\varepsilon}(s)-u_{\mu}(s)\big)\big\|_{L^2}^4|z|^2_{\R^m}\;\nu(\d z)\d s\Big)^{\frac12}\nonumber\\
&\lesssim  \EE\sup_{t\in[0,T]} \big\| \zeta_R\big( u_{\varepsilon}(t)-u_{\mu}(t)\big)\big\|_{L^2} \Big(\int_0^T  \big\|\zeta_R\big( u_{\varepsilon}(s)-u_{\mu}(s)\big)\big\|_{L^2}^2 \int_B |z|^2_{\R^m} \nu(\d z)\d s\Big)^{\frac12}\nonumber\\
&\leq \frac12 \EE\sup_{t\in[0,T]} \big\| \zeta_R\big( u_{\varepsilon}(t)-u_{\mu}(t)\big)\big\|_{L^2}^2+C(m) \int_0^T \EE\sup_{t\in[0,s]}\big\|\zeta_R\big( u_{\varepsilon}(t)-u_{\mu}(t)\big)\big\|_{L^2}^2\d s.\label{proof-main-eq101}
\end{align}

Let us fix $z \in \mathbb{R}^m$ and $y_1,y_2 \in \mathbb{C}$. Then straightforward calculations using the definition of $\Phi$ yields 
\begin{align}
&\frac{\partial}{\partial \theta}| \Phi(\theta, z, y_1)-\Phi(\theta, z, y_2)|^2\nonumber \\
&=2\operatorname{Re}\Big[\frac{\partial}{\partial \theta} \big(\Phi(\theta, z, y_1)-\Phi(\theta, z, y_2)\big)\overline{ \Phi(\theta, z, y_1)-\Phi(\theta, z, y_2)}\Big]\nonumber \\
& =-2\operatorname{Re}\Big[ \Big(\mathrm{i} \sum_{j=1}^m z_j g_j(\Phi(\theta, z, y_1))-\mathrm{i} \sum_{j=1}^m z_j g_j(\Phi(\theta, z, y_2))\Big)\overline{ \Phi(\theta, z, y_1)- \Phi(\theta, z, y_2)}\Big] .\label{proof-main-est-phi1}
\end{align}
By Condition \eqref{Ass-g-boundedness-1} and \eqref{id-g-deri-j-k}, we infer
\begin{align}
|g_j(x)-g_j(y)|  &\lesssim_{K_{\tilde{g}}}|x-y|,\label{g-Lip-est-1}\\
\big| (\mathrm{i} g_j)'[x](\mathrm{i}g_k(x))- (\mathrm{i} g_j)'[y](\mathrm{i}g_k(y))\big|  & \lesssim_{K_{\tilde{g}}} |x-y|.\label{gj-gk-est-lip}
\end{align}
We employ \eqref{Phi-est-Lip} and  \eqref{gj-gk-est-lip} to get
\begin{align}
&\Big|\int_0^{\theta}\sum_{j=1}^m z_j  \frac{\partial}{\partial a} \big(\mathrm{i}g_j(\Phi(a, z, y_1))-\mathrm{i}g_j(\Phi(a, z, y_2)\big)\d a\Big|\nonumber\\
&=\Big|\int_0^{\theta}\sum_{j=1}^m z_j \frac{\mathrm{d}(\mathrm{i} g_j)}{\mathrm{d} \Phi}(\Phi(a, z, y_1))\Big(-\mathrm{i} \sum_{k=1}^m z_k\big( g_k(\Phi(a, z, y_1))\Big)\nonumber\\
&\quad\quad-\sum_{j=1}^m z_j\frac{\mathrm{d}(\mathrm{i} g_j)}{\mathrm{d} \Phi}(\Phi(a, z, y_2)) \Big(-\mathrm{i} \sum_{k=1}^m z_k g_k(\Phi(a, z, y_2))\Big) \d a\Big|\nonumber\\
& \leq|z|_{\mathbb{R}^m}^2 \int_0^1 \int_0^\theta\Big(\sum_{j=1}^m \sum_{k=1}^m \Big\lvert\, \frac{\mathrm{d}(\mathrm{i} g_j)}{\mathrm{d} \Phi}(\Phi(a, z, y_1))(-\mathrm{i})   g_k(\Phi(a, z, y_1))\nonumber\\
&\quad\quad\quad\quad-\frac{\mathrm{d}(\mathrm{i} g_j)}{\mathrm{d} \Phi}(\Phi(a, z, y_2))(-\mathrm{i})  g_k(\Phi(a, z, y_2)))\Big|^2 \Big)^{\frac{1}{2}} \mathrm{~d} a\nonumber\\
& \leq |z|_{\mathbb{R}^m}^2 \int_0^\theta\Big(\sum_{j=1}^m \sum_{k=1}^m\left|\Phi\left(a, z, y_1\right)-\Phi\left(a, z, y_2\right)\right|^2\Big)^{\frac{1}{2}} \mathrm{~d} a \nonumber\\
& \lesssim_{m,K_{\tilde{g}}}|z|_{\mathbb{R}^m}^2|y_1-y_2|.\label{proof-mian-est-phi-g-1}
\end{align}
Taking account the identity \eqref{proof-main-est-phi1} and  estimates \eqref{g-Lip-est-1},   \eqref{Phi-est-Lip}, \eqref{G-est-Lip-1} and \eqref{proof-mian-est-phi-g-1}, we deduce that
\begin{align*}
&\Big||\Phi(z,y_1)-\Phi(z,y_2)|^2-|y_1-y_2|^2+2\mathrm{Re}(\overline{y_1-y_2})\big(\mathrm{i} \sum_{j=1}^mz_j(g_j(y_1)-g_j(y_2))\big)\Big|\\
&=\Big|\int_0^1 \frac{\partial }{\partial \theta} \big(|\Phi(\theta, z, y_1)-\Phi(\theta, z, y_2)|^2\big) +2\mathrm{Re}(\overline{y_1-y_2})\big( \mathrm{i}\sum_{j=1}^mz_j(g_j(y_1)-g_j(y_2))\big) \d \theta\Big|\\
&=\Big|\int_0^1 -2\sum_{j=1}^m z_j \operatorname{Re}\overline{(\Phi(\theta, z, y_1)- \Phi(\theta, z, y_2))}\big(\mathrm{i}g_j(\Phi(\theta, z, y_1))-\mathrm{i}g_j(\Phi(\theta, z, y_2))\big) \\
&\quad\quad +2\sum_{j=1}^mz_j\mathrm{Re}(\overline{y_1-y_2})\big( (\mathrm{i}g_j(y_1)-\mathrm{i}g_j(y_2))\big) \d \theta\Big|\\
&\lesssim \int_0^1  m^{\frac12}|z|_{\R^m}\big| (\Phi(\theta, z, y_1)- \Phi(\theta, z, y_2)-(y_1-y_2))\big| \big|g_j(\Phi(\theta, z, y_1))-g_j(\Phi(\theta, z, y_2)) \big| \d \theta\\
&\quad+\int_0^1 \Big| \sum_{j=1}^m z_j \operatorname{Re}\overline{(y_1-y_2)}\big(\int_0^{\theta} \frac{\partial}{\partial a} \big(\mathrm{i}g_j(\Phi(a, z, y_1))-\mathrm{i}g_j(\Phi(a, z, y_2)\big)\d a\big)\Big|  \d \theta\\
&\lesssim_{m,K_{\tilde{g}}} |z|^2_{\R^m}|y_1-y_2|^2.
\end{align*}
Based on the above estimation, for the last term on the right hand side of \eqref{proof-mian-est-cau-eq11} we derive
\begin{align}
&\EE \sup_{0\leq t\leq T}\Big|\int_0^t\int_B \big\| \zeta_R\big( \Phi(z, u_{\varepsilon}(s)) - \Phi(z, u_{\mu}(s)\big)\big\|_{L^2}^2- \big\| \zeta_R\big( u_{\varepsilon}(s)-u_{\mu}(s)\big)\big\|_{L^2}^2\nonumber \\
&\quad \quad\quad+2\big\langle u_{\varepsilon}(s)-u_{\mu}(s), \mathrm{i} \sum_{j=1}^m z_j \zeta_R^2\big(g_j(u_{\varepsilon}(s))- g_j(u_{\mu}(s))\big) \big\rangle_{L^2} \nu(\d z)\d s\Big|\nonumber\\
&\lesssim_{m,K_{\tilde{g}}} \EE \int_0^T \int_B  \big\|\zeta_R\big( u_{\varepsilon}(s)-u_{\mu}(s)\big)\big\|_{L^2}^2 |z|^2_{\R^m}  \nu(\d z)\d s.\label{proof-main-eq102}
\end{align}
Applying the embedding Gagliardo-Nirenberg interpolation inequality, 
\begin{align*}
\|u\|_{L^q} \leq C\|u\|^{1-\gamma}\|\nabla u\|^\gamma ,
\end{align*}
with $\frac{1}{q}=\frac{1}{2}-\frac{\gamma}{d}$ and $\gamma \in[0,1]$, shows
\begin{align}\label{proof-main-eq103}
\|u\|_{L^{2+2\delta}} \leq C\|u\|_{L^2}^{1-\frac{\delta d}{2+2\delta}}\|\nabla u\|_{L^2}^{\frac{\delta d}{2+2\delta}},
\end{align}
where $0<\delta\leq \frac{2}{d-2}$. 
By the mass conservation, we find
\begin{align}\label{proof-main-eq104}
\left\|\zeta_R^2\left(u_{\varepsilon}(s)-u_\mu(s)\right)\right\|_{L^1} \leq\|u_{\varepsilon}(s)-u_\mu(s)\|_{L^2(B_{2 R})}|B_{2 R}|^{1 / 2} \leq 2|B_{2 R}|^{1 / 2}\|u_0\|_{L^2} .
\end{align}
For $p\geq 2$, denote $M_T^p:=\sup_{\kappa\in(0,1]}\EE\sup_{t\in[0,T]} \|u_{\kappa}(t)\|_{H^1}^p<\infty$ by \eqref{uni-bound-H1-eq}. 
Therefore, combining \eqref{proof-mian-est-cau-eq11}, \eqref{proof-main-eq101}, \eqref{proof-main-eq102}, \eqref{proof-main-eq103} and \eqref{proof-main-eq104} yields
\begin{align*}
& \EE\sup_{t\in[0,T]}\left\|\zeta_R\big( u_{\varepsilon}(t)-u_{\mu}(t)\big)\right\|^2_{L^2} \\
&\leq
\frac{C}{R}T\Big(M_T^2\EE\|u_0\|^2_{L^2}\Big)^{\frac{1}{2}}+C (1-\varepsilon^2)|\lambda|\int_0^T \EE\sup_{r\in[0,s]} \|\zeta_R\big(u_{\varepsilon}(r)-u_{\mu}(r)\big) \|_{L^2}^2\d s\\
&+2TC|\lambda| |\varepsilon-\mu| |B_{2R}|^{\frac12}\EE\|u_0\|_{L^2}\\
&+C(\delta)|\lambda|  |\varepsilon-\mu|^{\delta}\EE\sup_{t\in[0,T]}\int_0^t \|\zeta_R(u_{\varepsilon}(s)-u_{\mu}(s))\|_{L^2}\|u_{\varepsilon}(s)\|_{L^2}^{(1-\frac{\delta d}{2+2\delta})(1+\delta)}\|\nabla u_{\varepsilon}(s)\|_{L^2}^{\frac{\delta d}{2}}\,\d s\\
&+ \frac12 \EE\sup_{t\in[0,T]} \big\| \zeta_R\big( u_{\varepsilon}(t)-u_{\mu}(t)\big)\big\|_{L^2}^2+C(m) \int_0^T \EE\sup_{r\in[0,s]}\big\|\zeta_R\big( u_{\varepsilon}(r)-u_{\mu}(r)\big)\big\|_{L^2}^2\d s\\
&+C(m,K_{\tilde{g}} )\,\EE \int_0^T \int_B  \big\|\zeta_R\big( u_{\varepsilon}(s)-u_{\mu}(s)\big)\big\|_{L^2}^2 |z|^2_{\R^m}  \nu(\d z)\d s,
\end{align*}
hence
\begin{align*}
 \EE\sup_{t\in[0,T]}\left\|\zeta_R\big( u_{\varepsilon}(t)-u_{\mu}(t)\big)\right\|_{L^2}^2 
&\leq \frac{C}{R}T\Big(M_T^2\EE\|u_0\|^2_{L^2}\Big)^{\frac{1}{2}}+C (1-\varepsilon^2)|\lambda|\int_0^T \EE\sup_{t\in[0,s]} \|\zeta_R\big(u_{\varepsilon}(t)-u_{\mu}(t)\big) \|_{L^2}^2\d s\\
&+2CT|\lambda| |\varepsilon-\mu| |B_{2R}|^{\frac12}\EE\|u_0\|_{L^2}\\
&+C(\delta)|\lambda|  |\varepsilon-\mu|^\delta TM_T^{{2+\delta}}\\
&+C(m)  \int_0^T \EE\sup_{r\in[0,s]}\big\|\zeta_R\big( u_{\varepsilon}(r)-u_{\mu}(r)\big)\big\|_{L^2}^2\d s\\\
&+C(m,K_{\tilde{g}} )  \int_0^T \EE\sup_{r\in[0,s]}\big\|\zeta_R\big( u_{\varepsilon}(r)-u_{\mu}(r)\big)\big\|_{L^2}^2\d s.
\end{align*}
By Gronwall's Lemma, we infer
\begin{align*}
&\EE\sup_{t\in[0,T]}\left\|\zeta_R\big( u_{\varepsilon}(t)-u_{\mu}(t)\big)\right\|_{L^2}^2 \\
 &\lesssim_{|\lambda|,m,K_{\tilde{g}},T} \frac{1}{R}\Big(M_T^2\EE\|u_0\|^2_{L^2}\Big)^{\frac{1}{2}}+ |\mu-\varepsilon| |B_{2R}|^{\frac12}\EE\|u_0\|_{L^2}+C(\delta)|\lambda|  |\varepsilon-\mu|^\delta TM_T^{{2+\delta}}.
\end{align*}
We now fix $R_0>0$ and take $R \in\left(R_0, \infty\right)$ as a parameter. It follows that
\begin{align*}
\limsup _{\mu, \varepsilon \downarrow 0}\EE\sup_{t\in[0,T]}\left\| u_{\varepsilon}(t)-u_{\mu}(t)\right\|_{L^2(B_{R_0})}^2
&\leq  \limsup _{\mu, \varepsilon \downarrow 0} \EE\sup_{t\in[0,T]}\left\|\zeta_R\big( u_{\varepsilon}(t)-u_{\mu}(t)\big)\right\|_{L^2}^2\\
&\lesssim_{|\lambda|,m,K_{\tilde{g}},T} \frac{1}{R}T\Big(M_T^2\EE\|u_0\|^2_{L^2}\Big)^{\frac{1}{2}} \xrightarrow[R \rightarrow \infty]{} 0 .
\end{align*}
Since $R_0>0$ is arbitrary, we deduce that $\left\{u_{\varepsilon}\right\}_{0<\varepsilon<1}$ forms a Cauchy sequence of $$L^2(\Omega;L^{\infty}(0,T;L^2_{\text{loc}}(\R^d) )).$$


\textbf{Step 2}
Assume that $u_0\in L^p(\Omega;H^1(\R^d))$, $p\geq 2$. 

By using Proposition \ref{prop-boundedness-L2}, combining Step 1 and \eqref{uni-bound-H1-eq}, we deduce that there exists a subsequence (still denoted by $u_{\varepsilon}$) and $u\in L^p(\Omega;L^{\infty}(0,T;H^1(\R^d) ))$, $p\geq 2$ and $u\in \mathbb{D}(0,T; L^2_{\text{loc}}(\R^d))$, $\mathbb{P}$-a.s. such that
\begin{align} 
&u_{\varepsilon}\rightarrow u\quad \text{in}\quad L^2(\Omega;L^{\infty}(0,T;L^2_{\text{loc}}(\R^d) ))\quad \text{as }\varepsilon\rightarrow0;\label{Th-proof-con-L2-loc}\\
&u_{\varepsilon}\rightharpoonup^* u\quad \text{in}\quad L^p(\Omega;L^{\infty}(0,T;L^2(\R^d) ))\quad \text{as }\varepsilon\rightarrow0;\label{Th-proof-con-weak-L2}\\
&u_{\varepsilon}\rightharpoonup^* u\quad \text{in}\quad L^p(\Omega;L^{\infty}(0,T;H^1(\R^d) ))\quad \text{as }\varepsilon\rightarrow0.\label{Th-proof-con--weak-H1}
\end{align}

In view of the strong convergence \eqref{Th-proof-con-L2-loc} and the mass conservation \eqref{mass-conser-u-epsilon} of $u_{\varepsilon}$, we infer for any $R>0$,
\begin{align*} 
\sup_{0\leq t\leq T}\|u(t)\|_{L^2(B_{R})} &\leq  \liminf_{\varepsilon\rightarrow 0} \sup_{0\leq t\leq T}\|u_{\varepsilon}(t)-u(t)\|_{L^2(B_{R})}+\liminf_{\varepsilon\rightarrow 0} \sup_{0\leq t\leq T}\|u_{\varepsilon}(t)\|_{L^2(B_{R})} \\
&\leq 
\|u_0\|_{L^2(\R^d)},
\end{align*}
 $\mathbb{P}$-a.s. with $B_R=\{x\in\R^d:|x|<R\}$. Letting $R\rightarrow \infty$ gives 
\begin{align}\label{L2-boundedness-u-eq} 
\sup_{0\leq t\leq T}\|u(t)\|_{L^2} \leq \|u_0\|_{L^2}, \quad\mathbb{P}\text{-a.s.}
\end{align}
Let $\rho_{n}$ be a sequence of mollifiers on $\R^d$. For any $\phi\in L^2(\R^d)$,  we have $\rho_{n}\star \phi \in C_c^{\infty}(\R^d)$ and $\rho_{n}\star \phi \rightarrow \phi$ in $L^2(\R^d)$. 
It follows that 
\begin{align}
 & \EE\sup_{t\in[0,T]}  \Big| \int_{\R^d} u_{\varepsilon}(t)\phi-u(t)\phi\, \d x\Big|\nonumber\\
 & \leq   \EE\sup_{t\in[0,T]}\Big|\int_{\R^d} (u_{\varepsilon}(t)-u(t))(\rho_n\star\phi-\phi)\d x \Big|+  \EE\sup_{t\in[0,T]}\Big|\int_{\R^d} (u_{\varepsilon}(t)-u(t))\rho_n\star\phi\d x \Big|\nonumber\\
    &\leq   \EE\sup_{t\in[0,T]}\|u_{\varepsilon}(t)-u(t)\|_{L^2}\|\rho_n\star\phi-\phi \|_{L^2}+  \EE\sup_{t\in[0,T]}\Big|\int_{\R^d} (u_{\varepsilon}(t)-u(t))\rho_n\star\phi\,\d x \Big|\nonumber\\
    &\leq   \EE\sup_{t\in[0,T]}(\|u_{\varepsilon}(t)\|_{L^2}+\|u(t)\|_{L^2})\|\rho_n\star\phi-\phi \|_{L^2}+  \EE\sup_{t\in[0,T]}\Big|\int_{\R^d} (u_{\varepsilon}(t)-u(t))\rho_n\star\phi\,\d x \Big|\nonumber\\
    &\rightarrow 0,\label{weak-con-u-varepsilon-L2}
\end{align}
where we used \eqref{mass-conser-u-epsilon}, $u\in L^2(\Omega;L^{\infty}(0,T;L^2(\R^d) ))$, and \eqref{Th-proof-con-L2-loc}. 

Similarly, since $u\in L^2(\Omega;L^{\infty}(0,T;H^1(\R^d) ))$, we have
\begin{align}
&\EE\sup_{0\leq t\leq T}\Big|\int_{\R^d} \nabla (u_{\varepsilon}(t)-u(t) )  \phi\, \d x\Big|\nonumber \\
&\leq \EE\sup_{0\leq t\leq T}\Big|\int_{\R^d} \nabla (u_{\varepsilon}(t)-u(t) )  (\rho_n\star\phi-\phi)\, \d x\Big| +\EE\sup_{0\leq t\leq T}\Big|\int_{\R^d} \nabla (u_{ \varepsilon}(t)-u(t) )\rho_n\star\phi \, \d x\Big|\nonumber\\
&\leq \EE\sup_{0\leq t\leq T}\|\nabla (u_{\varepsilon}(t)-u(t) )\|_{L^2}  \| \rho_n\star\phi-\phi\|_{L^2} +\EE\sup_{0\leq t\leq T} \|  u_{\varepsilon}(t)- u(t) \|_{L^2_{\text{loc}}}\|\nabla\rho_n\star\phi\|_{L^2}\nonumber\\
&\leq  \EE\sup_{0\leq t\leq T}(\|\nabla u_{\varepsilon}(t)\|_{L^2}+\|\nabla u(t) )\|_{L^2})  \| \rho_n\star\phi-\phi\|_{L^2}+\EE\sup_{0\leq t\leq T} \|  u_{\varepsilon}(t)- u(t) \|_{L^2_{\text{loc}}}\|\nabla\rho_n\star\phi\|_{L^2}\nonumber\\
&\rightarrow 0,\quad \text{as }n\rightarrow \infty,\; \varepsilon\rightarrow 0,\label{weak-con-u-epsi-H1}
\end{align}
where we also used \eqref{uni-bound-H1-eq} and \eqref{Th-proof-con-L2-loc}.

\textbf{Step 3. }Now we shall prove for all $t \in [0,T]$ 
\begin{align}\label{TH-proof-g-epsilon-con}
u_{\varepsilon} L_{\varepsilon}(u_{\varepsilon}) \rightarrow u \log |u|  \text { in } L^2(\Omega;L^{\infty}(0,T;L^2_{\text{loc}}(\mathbb{R}^d))) \quad \text { as } \varepsilon \downarrow 0.
\end{align}
 That is for any $D \subset \subset \mathbb{R}^d$
\begin{align*}
u_{\varepsilon} L_{\varepsilon}(u_{\varepsilon}) \rightarrow u \log |u| \quad \text { in } L^2(\Omega;L^{\infty}(0,T;L^2(D))) \quad \text { as } \varepsilon \downarrow 0.
\end{align*}

By Lemma \ref{lem-property-L_varepsilon},  for any small $\delta>0$ there exists $C(\delta)>0$ such that
\begin{align*}
|u_\varepsilon(t) L_\varepsilon(u_\varepsilon(t))-u(t) \log | u(t)|| \leq & \varepsilon+C(\delta)\varepsilon^{\delta}|u_\varepsilon(t)|^{1+\delta}+|u_\varepsilon(t)-u(t)|\\
&\quad+C(\delta) \big(1+|u_\varepsilon(t)|^{\frac12} \log ^{+}|u_\varepsilon(t)|+|u(t)|^{\frac12} \log ^{+}|u(t)|\big)|u_\varepsilon(t)-u(t)|^\frac12\\
\leq &\varepsilon+C(\delta)\varepsilon^{\delta}|u_\varepsilon(t)|^{1+\delta}+|u_\varepsilon(t)-u(t)|\\
&\quad+C(\delta)\big(1+|u_\varepsilon(t)|^{\frac12+\delta} +|u(t)|^{\frac12+\delta}\big)|u_\varepsilon(t)-u(t)|^\frac12,
\end{align*}
where we used $\log ^{+}|u(t)|  =\max \{0, \log |u(t)|\} \leq C(\delta) |u(t)|^\delta$, since $ \log ^{+}|u(t)| =\log |u(t)| \leqslant C(\delta) |u(t)|^\delta$, when $|u(t)|>1$. Fix $\delta>0$ satisfying $2+4 \delta<2^*=\frac{2d}{d-2}$. Then, we have
\begin{align*}
&\Big\||u_\varepsilon(t)|^{\frac{1}{2}+\delta}|u_\varepsilon(t)-u(t)|^{1 / 2}\Big\|_{L^2(D)}^2 
 \leq\Big(\int_{\R^d}|u_\varepsilon(t)|^{2+4 \delta}\d x\Big)^{1 / 2}\|u_\varepsilon(t)-u(t)\|_{L^2(D)}\\ 
 &\ \ \ \ \ \ \ \  \ \ \ \ \ \ \ \ 
 \ \ \ \ \ \ \ \ \ \ \ \ \ \ \ \  \ \ \ \ \ \ \ \ 
 \ \ \ \ \lesssim \|u_\varepsilon(t)\|_{H^1}^{1+2\delta}\|u_\varepsilon(t)-u(t)\|_{L^2(D)} ,\\
&\Big\||u(t)|^{\frac{1}{2}+\delta}|u_\varepsilon(t)-u(t)|^{1 / 2}\Big\|_{L^2(D)}^2 \lesssim \|u(t)\|_{H^1}^{1+2\delta}\|u_\varepsilon(t)-u(t)\|_{L^2(D)},
\end{align*}
where we used the embedding that $H^1(\R^d) \hookrightarrow  L^q(\R^d)$, for $2\leq q< 2^*$. Then the convergence \eqref{TH-proof-g-epsilon-con} follows from  \eqref{uni-bound-H1-eq}, \eqref{Th-proof-con-L2-loc} and the fact that $u\in L^2(\Omega;L^{\infty}(0,T;H^1(\R^d) ))$.



\textbf{Step 4} In this step, we shall verify that the process $u$ satisfies \eqref{SLSE-1} in $H^{-1}(D)$ $\mathbb{P}$-a.s. for every bounded open set $D\subset \R^d$. In view of Propositions \ref{prop-boundedness-L2} and \ref{prop-equv-two-solutions}, for every $\psi \in C_c^{\infty}(\mathbb{R}^d)$ and  $\phi \in C^1_c([0,T])$, we have
\begin{align*}
& \int_{0}^T\big\langle u_{\varepsilon}(t), \psi  \phi(t)\big\rangle_{L^2} \d t\\
    &=\int_{0}^T\,_{H^{-1}}\big\langle  \int_0^t \mathrm{i}\Delta u_{\varepsilon}(s)+2 \mathrm{i}\lambda u_{\varepsilon}(s)L_{\varepsilon}(u_{\varepsilon}(s))\d s, \psi\phi(s)\big\rangle_{ H^1}  \d t\\
    &\quad +\int_{0}^T\,_{H^{-1}}\big\langle \int_0^t \int_B \Phi(z, u_{\varepsilon}(s-))-u_{\varepsilon}(s-) \tilde{N}(\d s,\d z),\psi \phi(s-)\big\rangle_{H^1} \d t\\
      &\quad+\int_{0}^T\,_{H^{-1}}\big\langle \int_0^t  \int_B \Phi(z, u_{\varepsilon}(s))-u_{\varepsilon}(s)+\mathrm{i} \sum_{j=1}^m z_j g_j(u_{\varepsilon}(s)), \psi\phi(s)\big\rangle_{H^1}  \d t\\
        &=-\int_{0}^T \int_0^t \left\langle\mathrm{i}\nabla u_{\varepsilon}(s), \nabla \psi\phi(s)\right\rangle_{ L^2}\d s\,  \d t+\int_{0}^T \int_0^t \left\langle2 \mathrm{i}\lambda u_{\varepsilon}(s)L_{\varepsilon}(u_{\varepsilon}(s)), \psi\phi(s)\right\rangle_{ L^2}\d s\,  \d t\\
    &\quad+\int_{0}^T \int_0^t \int_B\left\langle  \Phi(z, u_{\varepsilon}(s-))-u_{\varepsilon}(s-) ,\psi\phi(s-) \right\rangle_{L^2} \tilde{N}(\d s,\d z)\,  \d t\\
      &\quad+\int_{0}^T \int_0^t \int_B \big\langle  \Phi(z, u_{\varepsilon}(s))-u_{\varepsilon}(s)+\mathrm{i} \sum_{j=1}^m z_j g_j(u_{\varepsilon}(s)), \psi \phi(s) \big\rangle_{L^2}  \nu(\d z)\d s\,  \d t.
\end{align*}
Let $K:=\text{supp}\, \psi$. In view of \eqref{weak-con-u-epsi-H1}, 
we obtain
\begin{align*}
&\mathbb{E}\int_0^T \Big|  \int_0^t \left\langle\mathrm{i}\nabla u_{\varepsilon}(s), \nabla \psi\phi(s)\right\rangle_{ L^2}\d s-  \int_0^t \left\langle\mathrm{i}\nabla u(s), \nabla \psi\phi(s)\right\rangle_{ L^2}\d s \Big| \d t\\
&\leq \mathbb{E} \sup_{s\in[0,T]}|\langle \mathrm{i} \nabla u_{\varepsilon} (s)- \mathrm{i} \nabla u(s),\nabla \psi\rangle_{L^2}|  \int_0^T  t    |\phi(t)|  \d t\rightarrow 0.
\end{align*}
Hence 
\begin{align*}
 \int_0^\cdot \left\langle\mathrm{i}\nabla u_{\varepsilon}(s), \nabla \psi\phi(s)\right\rangle_{ L^2}\d s\rightarrow  \int_0^\cdot \left\langle\mathrm{i}\nabla u(s), \nabla \psi\phi(s)\right\rangle_{ L^2}\d s
\end{align*}
 in $L^1(\Omega\times[0,T])$, as $\varepsilon\downarrow 0$.  Since by Step 3, 
 \begin{align*}
2 u_{\varepsilon}(\cdot) L_{\varepsilon} \left(u_{\varepsilon}(\cdot)\right) \rightarrow 2u(\cdot)\log |u(\cdot)| \text { in } L^2(\Omega;L^{\infty}(0,T;L^2_{\text{loc}}(\mathbb{R}^d))) \quad \text { as } \varepsilon \downarrow 0.
\end{align*}
Furthermore, the function $h(t,x)= \psi(x)\phi(t) $ has compact support. It follows that
 \begin{align*}
2 u_{\varepsilon}(\cdot) L_{\varepsilon} \left(u_{\varepsilon}(\cdot)\right)\psi\phi(\cdot) \rightarrow 2u(\cdot)\log |u(\cdot)|\psi\phi(\cdot) \text { in } L^2(\Omega;L^{\infty}(0,T;L^2(\mathbb{R}^d))) \quad \text { as } \varepsilon \downarrow 0,
\end{align*}
To prove the following convergences in $L^2(\Omega\times[0,T])$
\begin{align*}
&\int_{0}^\cdot\int_B \big\langle   \Phi(z, u_{\varepsilon}(s-))-u_{\varepsilon}(s-), \psi \phi(s-) \big\rangle_{L^2} \tilde{N}(\d s,\d z) \\
& \rightarrow \int_{0}^\cdot\int_B \big\langle   \Phi(z, u(s-))-u(s-), \psi \phi(s-) \big\rangle_{L^2} \tilde{N}(\d s,\d z),
\end{align*}
in view of the It\^o isometry, it is equivalent to prove 
\begin{align}\label{main-Th-proof-id-u-eq10}
& \int_0^T \int_B \big|\big\langle   \Phi(z, u_{\varepsilon}(s))-u_{\varepsilon}(s)-( \Phi(z, u(s))-u(s)), \psi \phi(s) \big\rangle_{L^2}\big|^2 \nu(\d z)\d s\rightarrow 0
\end{align}
in $L^1(\Omega\times[0,T])$, as $\varepsilon\downarrow0$. 
Observe that by \eqref{G-est-Lip-1}, we have 
\begin{align*}
&\big|\big\langle   \Phi(z, u_{\varepsilon}(s))-u_{\varepsilon}(s)-( \Phi(z, u(s))-u(s)), \psi \phi(s) \big\rangle_{L^2}\big|^2\\
&\leq \| \Phi(z, u_{\varepsilon}(s))-u_{\varepsilon}(s)-( \Phi(z, u(s))-u(s))\|_{L^2(K)}^2 \| \psi \phi(s) \|_{L^2}^2\\
&\lesssim_{K} \| u_{\varepsilon}(s) -u(s)\|_{L^2(K)}^2|z|_{\R^m}^2 \| \psi \phi \|_{L^\infty}^2.
\end{align*}
Here $\| \psi \phi \|_{L^\infty}=\sup_{x\in\R^d}\sup_{t\in[0,T]}|\psi(x) \phi(t)|<\infty$. Then the convergence \eqref{main-Th-proof-id-u-eq10} follows immediately from \eqref{Th-proof-con-L2-loc}. By applying  \eqref{H-est-Lip-1}, we arrive at the following estimate
\begin{align*}
&\Big|   \big\langle  \Phi(z, u_{\varepsilon}(s))-u_{\varepsilon}(s)+\mathrm{i} \sum_{j=1}^m z_j g_j(u_{\varepsilon}(s)), \psi \phi(s) \big\rangle_{L^2} -    \big\langle  \Phi(z, u(s))-u(s)+\mathrm{i} \sum_{j=1}^m z_j g_j(u(s)), \psi \phi(s) \big\rangle_{L^2}   \Big|\\
&\lesssim_{m,K_{\tilde{g}}, K}|z|_{\R^m}^2\|  u_{\varepsilon}-u \|_{L^2(K)}\|\psi \phi \|_{L^\infty}.
\end{align*} Hence we can derive that in $L^1(\Omega\times[0,T])$
\begin{align*}
&\int_0^\cdot \int_B \big\langle  \Phi(z, u_{\varepsilon}(s))-u_{\varepsilon}(s)+\mathrm{i} \sum_{j=1}^m z_j g_j(u_{\varepsilon}(s)), \psi \phi(s) \big\rangle_{L^2}\,  \nu(\d z)\d s \\
&\rightarrow \int_0^\cdot \int_B \big\langle  \Phi(z, u(s))-u(s)+\mathrm{i} \sum_{j=1}^m z_j g_j(u(s)), \psi \phi(s) \big\rangle_{L^2}  \,\nu(\d z)\d s,\quad \text{as}\quad \varepsilon\downarrow0.
\end{align*}
All of the above results in this step imply that
\begin{align*}
   & \int_{0}^T\left( u(t), \psi  \phi(t)\right)_{L^2} \d t\\
        &= -\int_{0}^T \int_0^t \left\langle\mathrm{i}\nabla u(s), \nabla \psi\phi(s)\right\rangle_{ L^2}\d s  \d t+\int_{0}^T \int_0^t \left\langle2 \mathrm{i}\lambda u(s)\log |u(s)|, \psi\phi(s)\right\rangle_{ L^2}\d s\,  \d t\\
    &\quad+ \int_{0}^T \int_0^t \int_B\left\langle  \Phi(z, u(s-))-u(s-) ,\psi\phi(s-) \right\rangle_{L^2} \tilde{N}(\d s,\d z)\,  \d t\\
      &\quad+ \int_{0}^T \int_0^t \int_B \big\langle  \Phi(z, u(s))-u(s)+\mathrm{i} \sum_{j=1}^m z_j g_j(u(s), \psi \phi(s) \big\rangle_{L^2}  \nu(\d z)\d s\,  \d t.
\end{align*}
Let 
\begin{align*}
   \hat{u}(t)
           &:=\int_0^t \mathrm{i}\Delta u(s) \d s +\int_0^t  \mathrm{i}\lambda u(s)\log |u(s)|^2 \d s + \int_0^t \int_B  \Phi(z, u(s-))-u(s-)  \tilde{N}(\d s,\d z) \\
      &\quad + \int_0^t \int_B \Phi(z, u(s))-u(s)+\mathrm{i} \sum_{j=1}^m z_j g_j(u(s))\nu(\d z)\d s.
\end{align*}
Then $\hat{u}\in \mathbb{D}(0,T; H^{-1}(D))$ and $u=\hat{u}$ in $L^2(\Omega\times[0,T]; H^{-1}(D))$. 
Recall that $u\in \mathbb{D}(0,T; L^{2}_{\text{loc}}(\mathbb{R}^d))$, $\mathbb{P}$-a.s. 
Hence we obtain
\begin{align*}
        \mathbb{P}\{ u(t)=\hat{u}(t)\quad\text{ for all }t\in[0,T]  \}=1.
\end{align*}
Hence in $H^{-1}(D)$ $\mathbb{P}$-a.s.
\begin{align}
    u(t)
          &=\int_0^t \mathrm{i}\Delta u(s) \d s +  \int_0^t  \mathrm{i}\lambda u(s)\log |u(s)|^2 \d s+ \int_0^t \int_B  \Phi(z, u(s-))-u(s-)  \tilde{N}(\d s,\d z) \nonumber\\
      &+ \int_0^t \int_B \Phi(z, u(s))-u(s)+\mathrm{i} \sum_{j=1}^m z_j g_j(u(s))\nu(\d z)\d s .\label{eq-proof-u-H-1}
\end{align}

\textbf{Step 5} Let us now prove the uniqueness of the solution. 
 Assume that $u,v\in L^2(\Omega;L^{\infty}(0,T;H^1(\R^d))$ be two solutions to \eqref{SLSE-1}. Set $$M_T:=\max\{\mathbb{E}\sup_{0\leq t\leq T}\|u(t)\|^2_{H^1},\mathbb{E}\sup_{0\leq t\leq T}\|v(t)\|^2_{H^1}\}<\infty.$$

 Recall  \eqref{L2-boundedness-u-eq} and $\zeta_R$ in Step 1. Arguments similar to the  proof given in Step 1, showing that the sequence $\{u_{\varepsilon}\}_{0<\varepsilon<1}$ forms a Cauchy sequence, gives
\begin{align*}
  \left\|\zeta_R\big( u(t)-v(t)\big)\right\|_{L^2}^2= & \int_0^t2 \langle \nabla(\zeta_R^2) \nabla(u(s)-v(s)), \mathrm{i}( u(s)-v(s))\rangle_{L^2}\d s\\
&+\int_0^t 2\left\langle  u(s)-v(s), \mathrm{i}2 \lambda\zeta_R^2\big(u(s) \log|u(s)|- v(s) \log|v(s)|\big) \right\rangle_{L^2} d s\\
&+\int_0^t\int_B \big\| \zeta_R \big(\Phi(z, u(s-)) - \Phi(z, v(s-))\big)\big\|_{L^2}^2- \big\| \zeta_R\big(u(s-)-v(s-)\big)\big\|_{L^2}^2   \tilde{N}(\mathrm{d} s, \mathrm{d} z) \\
&+\int_0^t\int_B \big\| \zeta_R\big(\Phi(z, u(s)) - \Phi(z, v(s))\big)\big\|_{L^2}^2- \big\| \zeta_R\big(u(s)-v(s)\big)\big\|_{L^2}^2 \\
&\quad \quad\quad-2 \langle u(s)-v(s),\zeta_R^2\big(\Phi(z, u(s))-\Phi(z, v(s))-\big(u(s)-v(s)\big)\big)\rangle_{L^2}  \nu(\d z)\d s\\
&+\int_0^t\int_B 2\big\langle u(s)-v(s),\zeta_R^2\big(\Phi(z, u(s))-\Phi(z, v(s))-\big(u(s)-v(s)\big)\big)\rangle_{L^2}\\
&\quad \quad\quad+2\big\langle u(s)-v(s), \mathrm{i} \sum_{j=1}^m z_j \zeta_R^2 \big(g_j(u(s))- g_j(v(s))\big) \big\rangle_{L^2} \nu(\d z)\d s.
\end{align*}
Taking inequality \eqref{quasi-mon-intro} into account, we infer for $\delta\in(0,1)$
\begin{align*}
&\big\langle u(s)-v(s), \mathrm{i}2 \lambda \zeta_R^2\big(u(s)\log |u(s)| - v(s) \log|v(s)|\big) \big\rangle_{L^2}\\
&=-2\lambda \int_{\R^d} \operatorname{Im} \overline{ u(s)-v(s) }\zeta_R^2\big(u(s)\log |u(s)| - v(s) \log|v(s)|\big) \d x\\
&\lesssim |\lambda|  \|\zeta_R(u-v )\|_{L^2}^2.
\end{align*}
Similarly considerations as in Step 1, we have
\begin{align*}
& \EE\sup_{t\in[0,T]}\left\|\zeta_R\big( u(t)-v(t)\big)\right\|^2_{L^2} \\
&\leq
\frac{C}{R}T\Big(M_T\EE\|u_0\|^2_{L^2}\Big)^{1/2}+C |\lambda|\int_0^T \EE\sup_{t\in[0,s]} \|\zeta_R\big(u(t)-v(t)\big) \|_{L^2}^2\d s\\
&+ \frac12 \EE\sup_{t\in[0,T]} \big\| \zeta_R\big( u(t)-v(t)\big)\big\|_{L^2}^2+C(m) \int_0^T \EE\sup_{t\in[0,s]}\big\|\zeta_R\big( u(t)-v(t)\big)\big\|_{L^2}^2\d s\\
&+C(m,K_{\tilde{g}} )\,\EE \int_0^T \int_B  \big\|\zeta_R\big( u(s)-v(s)\big)\big\|_{L^2}^2 |z|^2_{\R^m}  \nu(\d z)\d s,
\end{align*}
By Gronwal's Lemma, we infer
\begin{align*}
\EE\sup_{t\in[0,T]}\left\|\zeta_R\big( u(t)-v(t)\big)\right\|_{L^2}^2 \lesssim_{|\lambda|,m,K_{\tilde{g}},T} \frac{1}{R}\Big(M_T\EE\|u_0\|^2_{L^2}\Big)^{1/2}.
\end{align*}
Applying Fatou's lemma, we obtain
$$
 \EE\sup_{t\in[0,T]}\|u(t)-v(t)\|_{L^2}^2\leq \liminf_{R\rightarrow\infty} \EE\sup_{t\in[0,T]}\left\|\zeta_R\big( u(t)-v(t)\big)\right\|_{L^2}^2= 0,
$$
which yields the uniqueness of the solution.

\textbf{Step 6} 
Assume that $u_0\in L^p(\Omega,\mathcal{F}_0;W)$. 
In view of Lemma \ref{lem-est-F-function}, we have for $p\geq 2$
 \begin{equation}
\mathbb{E}\left[\sup _{t \in[0, T]}\Big|  \int_{\R^d}  \big|F( |u^\varepsilon(t)|)  \big| \d x\Big|^p\right]\leq C(u_0,\lambda,m,T,p)<\infty.
\end{equation}
By using the Fatou lemma, we obtain
\begin{align*}
      \mathbb{E}\Big[\sup _{t \in[0, T]}\Big|  \int_{\R^d}  \big| F(|u(t)|) \big| \d x\Big|^p\Big]& \leq \liminf_{\varepsilon\rightarrow0}  \mathbb{E}\Big[\sup _{t \in[0, T]}\Big|  \int_{\R^d}  \big| F(|u^\varepsilon(t)|)  \big| \d x\Big|^p\Big]\\
      &\leq C(u_0,\lambda,m,T,p)<\infty.
\end{align*}
Set $B(s)=F(s)-N(s)$, $s>0$, where $F$ and $N$ are  defined in \eqref{defini-F-functiom} and \eqref{defini-A-functio} respectively. Note that for any $0<\delta<\frac{4}{d-2}$, we have
\begin{align*}
      \int_{\R^d} | B(|u(t)|)| \d x \lesssim_{\delta}  \int_{\R^d} |u(t)|^{2+\delta} \d x\lesssim \|u\|_{H^1}^{2+\delta},
\end{align*}
where we used again the embedding $H^1(\R^d) \hookrightarrow L^{2+\delta}(\mathbb{R}^d) $, $0<\delta<\frac{4}{d-2}$. 
It follows that
\begin{align*}
     \mathbb{E}\sup_{t\in[0,T]}\Big|  \int_{\R^d} N(|u(t)|) \d x\Big|^p &\lesssim_{p,\delta }  \mathbb{E}\sup_{t\in[0,T]}\Big|  \int_{\R^d} |F(|u(t)|)|\d x\Big|^p + \mathbb{E}\sup_{t\in[0,T]}\|u(t)\|_{H^1}^{(2+\delta)p}\\
     &\leq C(u_0,\lambda,m,T,p)<\infty.
\end{align*}
In view of \eqref{norm-in-Orlicz-space}, we infer for any $p\geq 2$,
\begin{align*}
   \mathbb{E}\sup_{t\in[0,T]}\|u(t)\|_{W}^p\leq C(u_0,\lambda,m,T,p)<\infty.
\end{align*}
Therefore, we deduce that
\begin{align*}
       u\in L^{\infty}(0,T;W)\quad\mathbb{P}\text{-a.s.}
\end{align*}
It follows from \eqref{eq-proof-u-H-1} and Lemma \ref{lem-Caz-bounded-L} that \eqref{eq-proof-u-H-1} holds in $W'$. 

Recall that we have the continuous embedding $W\hookrightarrow L^2(\R^d) \hookrightarrow W'$. Applying the It\^o formula (see e.g. \cite[Theorem A.1]{BHZ}), there exists a $\mathbb{P}$-full measure event   $\Omega'\subset\Omega$ with $\mathbb{P}(\Omega')=1$ such that on $\Omega'$, $u$ is an $L^2(\R^d)$-valued c\`adl\`ag process  and  we have $\mathbb{P}$-a.s. for every $t\in[0,T]$,
\begin{align*}
    \| u(t)\|^2_{L^2} 
       =&\|u(0)\|^2_{L^2}+ \int_0^t 2\,_{W}\langle u(s), \mathrm{i}\Delta u(s)+\lambda \mathrm{i}u(s)\log|u(s)|^2\rangle_{W'}\mathrm{~d} s \\
       &+\int_0^t \int_B\left[\| \Phi(z, u(s-))\|_{L^2}^2- \| u(s-)\|^2_{L^2}\right] \tilde{N}(\mathrm{d} s, \mathrm{d} z)\\
   &+\int_0^t \int_B\left[\| \Phi(z, u(s))\|^2_{L^2}-\| u(s)\|^2_{L^2} - \operatorname{Re} \langle u(s), \Phi(z, u(s))-u(s)\rangle_{L^2} \right] \nu(dz)ds\\
   &+ \int_0^t \int_B \Big(\operatorname{Re} \langle u(s), \Phi(z, u(s))-u(s)+\mathrm{i} \sum_{j=1}^m z_j g_j(u(s))  \rangle_{L^2}\Big)     \nu(dz)ds\\
    =&\|u(0)\|^2_{L^2},
\end{align*}
where we used the fact that $\operatorname{Re}(x\overline{\text{i}x})=0$, $\langle u(s), \text{i} \triangle u(s)\rangle_{L^2}=0$, and $\|\Phi(z, u)\|_{L^2}^2=\|u\|_{L^2}^2$, for $ u \in L^2(\mathbb{R}^d)$, $z \in \mathbb{R}^m$.

The proof of Theorem \ref{Them-main} is now complete.

\end{proof}

\section*{Appendix}\label{App-sec}

\begin{proof}[Proof of Lemma \ref{lem-property-L_varepsilon}]


Assertion (a)  is standard. Now let us prove assertion (b).
\begin{enumerate}
\item[(b)] Without loss of generality we may assume $0<\left|u_2\right| \leq\left|u_1\right|$. We first note that
\begin{align*}
u_1 L_{\varepsilon}\left(u_1\right)-u_2 L_{\varepsilon}\left(u_2\right)=u_2(L_{\varepsilon}(u_1)-L_{\varepsilon}(u_2))+(u_1-u_2) L_{\varepsilon}(u_1) .
\end{align*}
Since $\phi(x)=\frac{x+\varepsilon}{1+\varepsilon x}$ is increasing when $x>0$,  it follows that 
if $0<\left|u_2\right| \leq\left|u_1\right|$, we have $\phi(|u_2|)\leq\phi(|u_1|)$ and
$$\phi(|u_2|)\leq \delta \phi(|u_1|)+(1-\delta)\phi(|u_2|)\leq \phi(|u_1|).$$
Hence 
\begin{align}
\begin{split}\label{proof-est-L-eq-1}
\left|L_{\varepsilon}\left(u_1\right)-L_{\varepsilon}\left(u_2\right)\right| &=|\log(\phi(|u_1|))-\log(\phi(|u_2|))|\\
&=\frac1{\delta \phi(|u_1|)+(1-\delta)\phi(|u_2|)} |\phi(|u_1|)-\phi(|u_2|)|\\
&\leq \frac{1+\varepsilon\left|u_2\right|}{\left|u_2\right|+\varepsilon}\left|\frac{\left|u_1\right|+\varepsilon}{1+\varepsilon\left|u_1\right|}-\frac{\left|u_2\right|+\varepsilon}{1+\varepsilon\left|u_2\right|}\right| \\
& =\left|\frac{(1-\varepsilon^2)\left(|u_1|-|u_2|\right)}{(|u_2|+\varepsilon)(1+\varepsilon|u_1|)}\right| \\
& \leq\left(1-\varepsilon^2\right)\left|u_2\right|^{-1}\left|u_1-u_2\right|,
\end{split}
\end{align}
where we used the fact that $(|u_2|+\varepsilon)(1+\varepsilon|u_1|)> |u_2|$.

{The above result combining Assertion (a) yields Assertion (b).}

\item[(c)] Observe first that
\begin{align*}
&\operatorname{Im}(\overline{u_1}-\overline{u_2})(u_1 L_{\varepsilon}(u_1)-u_2 L_{\mu}(u_2))\\
&=\operatorname{Im}(\overline{u_1}-\overline{u_2})\big(u_1L_{\varepsilon}(u_1)-u_2L_{\varepsilon}(u_2))+u_2(L_{\varepsilon}(u_2)-L_{\mu}(u_2))\big).
\end{align*}
Then using the inequality
\begin{align*}
\operatorname{Im}(\overline{u_1} {u_2}) \leq \min \{|u_1|,|u_2|\}|u_1-u_2|
\end{align*}
and also \eqref{proof-est-L-eq-1}, we have
\begin{align*}
\Big|\operatorname{Im}\left[(\overline{u_1}-\overline{u_2})(u_1 L_{\varepsilon}(u_1)-u_2 L_{\varepsilon}(u_2))\right]\Big|&=\Big|\operatorname{Im}(\overline{u_1} u_2)(L_{\varepsilon}(u_1)-L_{\varepsilon}(u_2))\Big|\\
&\leq (1-\varepsilon^2)|u_1-u_2|^2.
\end{align*}
By applying the fact that $\log (1+x) \leq C(\delta) x^{\delta}$,  $\delta\in(0,1]$, $x\geq 0$, we infer
\begin{align*}
\left|L_{\varepsilon}\left(u_2\right)-L_{\mu}\left(u_2\right)\right| &=\Big| \log\big( \frac{\varepsilon+|u_2|}{1+\varepsilon|u_2|}\big)-\log\big( \frac{\mu+|u_2|}{1+\mu|u_2|}\big) \Big|\\
&=\Big| \log\big( \frac{\varepsilon+|u_2|}{\mu+|u_2|}\big)-\log\big( \frac{1+\varepsilon|u_2|}{1+\mu|u_2|}\big) \Big|\\
&\leq \Big| \log\big( 1+\frac{\varepsilon-\mu}{\mu+|u_2|}\big)\Big| +\Big| \log\big(1+ \frac{(\varepsilon-\mu)|u_2|}{1+\mu|u_2|}\big) \Big|\\
&\leq C \frac{| \varepsilon-\mu|}{\mu+|u_2|}+C(\delta)\Big|   \frac{(\varepsilon-\mu)|u_2|}{1+\mu|u_2|}\         \Big|^{\delta}\\
&\leq C| \varepsilon-\mu| \frac1{|u_2|} + C(\delta)| \varepsilon-\mu|^{\delta}  |u_2|^{\delta}.
\end{align*}
It follows that
\begin{align*}
&\left|\operatorname{Im}\left[(\overline{u_1}-\overline{u_2})(u_1 L_{\varepsilon}(u_1)-u_2 L_{\mu}(u_2))\right]\right|\\
&\leq (1-\varepsilon^2)|u_1-u_2|^2 +C |\varepsilon-\mu   ||u_1-u_2| +C(\delta)| \varepsilon-\mu|^{\delta}  |u_2|^{1+\delta} |u_1-u_2|,
\end{align*}
which proves assertion (c).
\item[(d)]We will examine two specific cases. \\
\textbf{Case 1}. Assume that $|u_1|\leq |u_2|$. 
Note that
\begin{align}
u_1 L_\varepsilon(u_1)-u_2 \log | u_2|=u_1( L_\varepsilon(u_1)-\log |u_2| )+(u_1-u_2)\log|u_2|.\label{est-L-esp-eq-11}
\end{align}
Since $|u_1|\leq |u_2|$, we estimate the first term on the right-hand side of \eqref{est-L-esp-eq-11} as 
\begin{align*}
| L_{\varepsilon}(u_1)- \log |u_2||&=\Big| \log \left(\frac{|u_1|+\varepsilon}{1+\varepsilon|u_1|}\right)- \log |u_2| \Big|\\
&\leq  \Big| \log \left(\frac{|u_1|+\varepsilon}{|u_2|}\right)  \Big| +  \Big|  \log (1+\varepsilon|u_1|) \Big|\\
&\leq \frac{|u_1-u_2|+\varepsilon}{\min\{|u_2|,|u_1|+\varepsilon\}}+C(\delta)\varepsilon^{\delta}|u_1|^{\delta}\\
&\leq \frac{|u_1-u_2|+\varepsilon}{|u_1|}+C(\delta)\varepsilon^{\delta}|u_1|^{\delta},
\end{align*}
where we used $\log (1+x) \leq C(\delta) x^{\delta}$,  $\delta\in(0,1]$, $x\geq 0$, here $C(1)=1$. Hence
\begin{align*}
| u_1(L_{\varepsilon}(u_1)- \log |u_2|)|\leq |u_1-u_2|+\varepsilon +C(\delta)\varepsilon^{\delta}|u_1|^{1+\delta}.
\end{align*}
For the second term on the right-hand side of \eqref{est-L-esp-eq-11}, we find, for any $\alpha\in(0,1)$,
\begin{align*}
|(u_1-u_2) \log | u_2|| 
& \leq C(\alpha)|u_1-u_2|^\alpha|u_2|^{1-\alpha} |\log |u_2| |
\end{align*}
Since $|u_2|^{1-\alpha} |\log |u_2|| $ is bounded when $0<|u_2|<1$, 
 we have $$|u_1-u_2|^\alpha|u_2|^{1-\alpha}| \log |u_2|| \leq C(\alpha) |u_1-u_2|^\alpha,\quad\text{for}\; |u_1|\leq |u_2|<1.$$
When $|u_2|\geq 1$, we find $|u_2|^{1-\alpha} |\log |u_2|| \leq |u_2|^{1-\alpha}\log^+|u_2|$.
It follows that
\begin{align*}
|(u_1-u_2) \log | u_2|| 
 \leq C(\alpha) |u_1-u_2|^\alpha\left(1+|u_2|^{1-\alpha} \log ^{+}|u_2|\right).
\end{align*}
Combining the above estimates yields \eqref{L-var-est-3}.\\

\textbf{Case 2}: $|u_2| \leq|u_1|$. In this case we write
\begin{align}
u_1 L_\varepsilon(u_1)-u_2 \log | u_2|=u_2( L_\varepsilon(u_1)-\log |u_2| )+ L_\varepsilon(u_1) (u_1-u_2).\label{est-L-esp-eq-12}
\end{align}
Similar as before, we can estimate the first term on the right-hand side of \eqref{est-L-esp-eq-12} as
\begin{align}
| L_{\varepsilon}(u_1)- \log |u_2||
&\leq \frac{|u_1-u_2|+\varepsilon}{|u_2|}+C(\delta)\varepsilon^{\delta}|u_1|^{\delta}\label{L-var-eq-proof-3}
\end{align}
and hence
\begin{align*}
| u_2(L_{\varepsilon}(u_1)- \log |u_2|)|
&\leq |u_1-u_2|+\varepsilon +C(\delta)\varepsilon^{\delta}|u_1|^{1+\delta}  .
\end{align*}

For the second term on the right-hand side of \eqref{est-L-esp-eq-12}, similar to Case 1, we consider two cases:
\begin{enumerate}
\item $|u_1|\geq 1$:  Since $0<\varepsilon<1$, we have  $L_{\varepsilon}(u_1)\geq0$ and 
 $\frac{\varepsilon+|u_1|}{1+\varepsilon|u_1|}\leq|u_1|^2$. It follows that 
\begin{align}\label{L-var-eq-proof-1}
L_{\varepsilon}(u_1)=\log\Big(\frac{\varepsilon+|u_1|}{1+\varepsilon|u_1|}\Big)\leq\log |u_1|^2=2\log^+|u_1|.
\end{align}
\item $|u_1|\leq 1$: Since $\varepsilon \mapsto L_{\varepsilon}(x)$ is increasing with $\varepsilon$ if $x \in(0,1]$, it follows that
\begin{align*}
\log x\leq L_{\varepsilon}(x)\leq0.
\end{align*}
Hence
\begin{align}\label{L-var-eq-proof-2}
 \Big||u_1|^{1-\alpha}L_{\varepsilon}(u_1)  \Big|\leq \big| |u_1|^{1-\alpha}\log |u_1|\big|\leq C(\alpha),
\end{align}
by the boundedness of $|x|^{1-\alpha}|\log|x||$ when $0<|x|\leq1$.
\end{enumerate}
Therefore, by using \eqref{L-var-eq-proof-1} and \eqref{L-var-eq-proof-2}, we obtain, for $|u_2| \leq|u_1|$, 
\begin{align}
 \Big| L_\varepsilon(u_1) (u_1-u_2) \Big|&\leq | L_\varepsilon(u_1)||u_1-u_2|^{1-\alpha}|u_1-u_2|^{\alpha}\nonumber\\
 & \leq  | L_\varepsilon(u_1)|(|u_1|+|u_2|)^{1-\alpha}|u_1-u_2|^{\alpha}\nonumber\\
 &\leq C(\alpha) |u_1|^{1-\alpha }| L_\varepsilon(u_1)| |u_1-u_2|^{\alpha}\nonumber\\
  &\leq C(\alpha)(1+ |u_1|^{1-\alpha}\log^+|u_1| ) |u_1-u_2|^{\alpha} \label{L-var-eq-proof-4}
 \end{align}
 Combining estimates \eqref{L-var-eq-proof-3} and \eqref{L-var-eq-proof-4} proves \eqref{L-var-est-3}.

\end{enumerate}
\end{proof}


\begin{proof}[Proof of Proposition \ref{prop-equv-two-solutions}]

Let $X=H^{m-2}(\R^d;\mathbb{C})$, $m=0,1$. It is known that $A=\mathrm{i}\Delta$ is $m$-dissipative in $X$ with dense domain $\mathcal{D}(A)=H^m(\R^d)$. 
As before, Let $(S_t)_{t \in \mathbb{R}}$ be the group generated by $A$ on $X$.

Note that $(S_t)_{t \in \mathbb{R}}$ can be extended to a group of isometries $(T_t)_{t \in \mathbb{R}}$ in $H^{m-4}(\mathbb{R}^d)$ which is the group generated by the operator $A_{-1}$ defined by 
\begin{align*}
\mathcal{D}(A_{-1})&=H^{m-2}(\mathbb{R}^d)\\
A_{-1} u&=\mathrm{i} \Delta u\quad \text{ for }u \in D(A_{-1}). 
\end{align*}
Moreover,  $A_{-1}$ coincides with $A$ on $H^{m}(\mathbb{R}^d)$ and $T_t$ coincides with $S_t$ on $H^{m-2}(\mathbb{R}^d)$. 

$(i) \Rightarrow (ii)$: 
Define a function $f \in C^{1,2}([0, t] \times H^{m-2}(\mathbb{R}^d); H^{m-4}(\mathbb{R}^d))$  by
\begin{align*}
f(s, x):=T_{t-s} x, \quad s \in[0, t], \quad x \in H^{m-2}(\mathbb{R}^d).
\end{align*}
Applying the It\^o formula to the function $f$ and the process $u$,  we infer
\begin{align*}
u(t)=&T_t u_0-\int_0^tA_{-1} T_{t-s} u(s) \mathrm{d} s+\int_0^t T_{t-s}[ A u(s)+2\lambda u(s) L_{\varepsilon}(u(s) )] \mathrm{d} s\\
&+\int_0^t\int_BT_{t-s}[\Phi(z, u(s-))-u(s-)]  \tilde{N}(\d s,\d z)\\
& +\int_0^t\int_BT_{t-s} \Big[\Phi(z, u(s))-u(s)+\mathrm{i} \sum_{j=1}^m z_j g_j(u(s))\Big] \nu(\mathrm{d} z) \mathrm{d} s\end{align*}
in $H^{m-4}(\mathbb{R}^d)$ for all $t\in[0,T]$ almost surely. Thanks to the c\`adl\`ag continuity of the processes, the null set can be chosen independently of $t \in[0, T]$. Since $u\in L(\Omega; L^\infty(0,T;H^m(\R^d)))$ and $|L_{\varepsilon}(u)| \leq|\log \varepsilon|$,  we infer $uL_{\varepsilon}(u) \in L(\Omega; L^\infty(0,T;H^{m}(\R^d)))$.  By assumption $u_0\in H^{m}(\R^d)$, we deduce that $u$ statisfies
\begin{align*}
u(t)
=&S_t u_0+\int_0^t S_{t-s}\big( 2\lambda u(s) L_{\varepsilon}(u(s) )\big)\mathrm{d} s+\int_0^t\int_B S_{t-s}[\Phi(z, u(s-))-u(s-)]  \tilde{N}(\d s,\d z)\\
& +\int_0^t\int_BS_{t-s} \Big[\Phi(z, u(s))-u(s)+\mathrm{i} \sum_{j=1}^m z_j g_j(u(s))\Big] \nu(\mathrm{d} z) \mathrm{d} s
\end{align*}
in $H^{m-2}(\mathbb{R}^d)$.

$(ii) \Rightarrow (i)$: Assume that \eqref{mild-solution-eq} holds.  Then we have
\begin{align*}
\int_0^t  A_{-1} u(s) \mathrm{d} s=&\int_0^t  A_{-1} S_s u_0 \mathrm{~d} s+\int_0^t  A_{-1} \int_0^s S_{s-r}  \big(2u(r)L_{\varepsilon}(u(r))\big) \mathrm{d} r \mathrm{~d} s \\
&+\int_0^t  A_{-1} \int_0^s\int_B S_{s-r} [\Phi(z, u(r-))-u(r-)] \tilde{N}(\d r,\d z)\;\d s\\
&+\int_0^t  A_{-1} \int_0^s\int_B S_{s-r}\Big[\Phi(z, u(r))-u(r)+\mathrm{i} \sum_{j=1}^m z_j g_j(u(r))\Big] \nu(\mathrm{d} z)\d r\;\d s,
\end{align*}
$\mathbb{P}$-a.s. in $H^{m-4}(\R^d)$ for all $t\in[0,T]$. Observe that

\begin{align*}
&\mathbb{E} \int_0^T\int_B \|  A_{-1}S_{s-r} [\Phi(z, u(r))-u(r)]  \|^2_{H^{m-4}(\R^d)}\nu(\d z)\d r\\
&\leq \mathbb{E} \int_0^T\int_B \|  S_{s-r} [\Phi(z, u(r))-u(r)]  \|^2_{H^{m-2}(\R^d)}\nu(\d z)\d r\\
&\leq \mathbb{E} \int_0^T\int_B \|  \Phi(z, u(r))-u(r) \|^2_{H^{m-2}(\R^d)}\nu(\d z)\d r\\
&\leq \mathbb{E} \int_0^T\int_B \|  \Phi(z, u(r))-u(r) \|^2_{H^{m}(\R^d)}\nu(\d z)\d r\\
&\leq  \mathbb{E} \int_0^T\int_B \|u(r)\|^2_{H^{m}(\R^d)}|z|^2_{\R^m}\nu(\d z)\d r<\infty,
\end{align*}
where we used the embedding $H^{m}(\R^d)\hookrightarrow H^{m-2}(\R^d)$ and Lemma \ref{lem-stoch-lip-1} when $m=0$ or Lemma \ref{lem-G-H1} when $m=1$ respectively.
Similarly, we have
\begin{align*}
&\mathbb{E} \int_0^T\int_B \|  A_{-1}S_{s-r}\Big[\Phi(z, u(r))-u(r)+\mathrm{i} \sum_{j=1}^m z_j g_j(u(r))\Big] \|_{H^{m-4}(\R^d)}\nu(\d z)\d r\\
&\leq \mathbb{E} \int_0^T\int_B \| S_{s-r}\Big[\Phi(z, u(r))-u(r)+\mathrm{i} \sum_{j=1}^m z_j g_j(u(r))\Big] \|_{H^{m-2}(\R^d)}\nu(\d z)\d r\\
&\leq \mathbb{E} \int_0^T\int_B \|  \Phi(z, u(r))-u(r)+\mathrm{i} \sum_{j=1}^m z_j g_j(u(r)) \|_{H^{m-2}(\R^d)}\nu(\d z)\d r\\
&\leq \mathbb{E} \int_0^T\int_B \| \Phi(z, u(r))-u(r)+\mathrm{i} \sum_{j=1}^m z_j g_j(u(r))  \|_{H^{m}(\R^d)}\nu(\d z)\d r\\
&\leq  \mathbb{E} \int_0^T\int_B \|u(r)\|_{H^{m}(\R^d)}|z|^2_{\R^m}\nu(\d z)\d r<\infty.
\end{align*}
By interchange $A_{-1}$ with the integrals and using the deterministic and stochastic Fubini Theorems, we have
\begin{align*}
\int_0^t  A_{-1} u(s) \mathrm{d} s=&\int_0^t  A_{-1} S_s u_0 \mathrm{~d} s+\int_0^t   \int_r^t A_{-1} S_{s-r} \big(2u(r)L_{\varepsilon}(u(r))\big) \mathrm{~d} s \mathrm{d} r  \\
&+\int_0^t \int_B\Big( \int_r^t A_{-1}  S_{s-r} [\Phi(z, u(r-))-u(r-)] \d s\;\Big) \tilde{N}(\d r,\d z)\\
&+\int_0^t \int_B\Big( \int_r^t A_{-1} S_{s-r}\Big[\Phi(z, u(r))-u(r)+\mathrm{i} \sum_{j=1}^m z_j g_j(u(r))\Big] \d s\;\Big)\nu(\mathrm{d} z) \d r\\
=&T_{t}u_0-u_0+\int_0^t  T_{t-r}\big(2u(r)L_{\varepsilon}(u(r))\big)-2u(r)L_{\varepsilon}(u(r))\d r\\
&+\int_0^t\int_B (T_{t-r}-I) [\Phi(z, u(r-))-u(r-)] \tilde{N}(\d r,\d z)\\
&+ \int_0^t \int_B(T_{t-r}-I)\Big[\Phi(z, u(r))-u(r)+\mathrm{i} \sum_{j=1}^m z_j g_j(u(r))\Big] \nu(\mathrm{d} z)\d r
\end{align*}
$\mathbb{P}$-a.s. in $H^{m-4}(\R^d)$ for all $t\in[0,T]$. It follows from \eqref{mild-solution-eq} that
\begin{align*}
u(t)=&u_0+\int_0^t A_{-1}u(s)\d s+ \int_0^t2u(s)L_{\varepsilon}(u(s))\d s+ \int_0^t \int_B [\Phi(z, u(s-))-u(s-)]\tilde{N}(\d s,\d z)\\
&+\int_0^t \int_B\Big[\Phi(z, u(s))-u(s)+\mathrm{i} \sum_{j=1}^m z_j g_j(u(s))\Big] \nu(\mathrm{d} z)\d s,
\end{align*}
$\mathbb{P}$-a.s. in $H^{-3}(\R^d)$ for all $t\in[0,T]$. Again, thanks to the c\`adl\`ag continuity of the processes, the null set can be chosen independently of $t \in[0, T]$. By the given assumption on $u_0$, $u$ and boundedness of $L_{\varepsilon}$ and Lemma  \ref{lem-G-H1}, we get \eqref{strong-solution-eq} as an equation in $H^{m-2}(\R^d)$.
\end{proof}

\end{document}